\documentclass{elsarticle} 
\usepackage{amsmath}
\usepackage{amssymb}
\usepackage{amsthm}
\usepackage{graphicx}
\usepackage{dcolumn}
\usepackage{bm}
\makeatletter
\newcommand{\bal}{\@ifstar{\@bals}{\@bal}}
\def\@bals#1\eal{\begin{align*}#1\end{align*}}
\def\@bal#1\eal{\begin{align}#1\end{align}}
\makeatletter

\makeatletter
\let\today\relax
\def\ps@pprintTitle{%
    \let\@oddhead\@empty
    \let\@evenhead\@empty
    \def\@oddfoot{\footnotesize\itshape
      \hfill\today}  
    \let\@evenfoot\@oddfoot
    }
\makeatother

\usepackage[utf8]{inputenc}
\usepackage[T1]{fontenc}
\usepackage{mathptmx}
\newtheorem{theorem}{Theorem}[section]
\newtheorem{remark}{Remark}[section]

\let\today\relax
\makeatletter
\def\ps@pprintTitle{%
    \let\@oddhead\@empty
    \let\@evenhead\@empty
    \def\@oddfoot{\footnotesize\itshape
         {Preprint to be published in Journal of Nonlinear Science} \hfill\today}%
    \let\@evenfoot\@oddfoot
    }
\makeatother

\def\ee{\end{eqnarray*}}
\def\be{\begin{eqnarray*}}
\def\bee{\end{eqnarray}}
\def\bbe{\begin{eqnarray}}
\def\ea{\end{align*}}
\def\ba{\begin{align*}}
\def\baa{\end{align}}
\def\bba{\begin{align}}

\def\Rh{\mathrm{Rh}}

    \def\p{\partial}

    \def\v{\bm v}
    \def\f{\bm f}

\begin{document}

\title{Steady-state bifurcation of  a non-parallel flow involving energy dissipation over a Hartmann boundary layer  }

\author{Zhi-Min Chen}
 \ead{zmchen@szu.edu.cn}

\address{School  of Mathematics and Statistics, Shenzhen University, Shenzhen 518060,  China}%


\begin{abstract}
A plane non-parallel vortex flow in a square fluid domain  is examined. The energy dissipation of the flow is dominated  by viscosity and  linear friction effect of a Hartmann layer. This is a traditional Navier-Stokes flow when the linear  friction effect  is not involved, whereas it is a magnetohydrodynamic flow when the energy dissipation is fundamentally dominated by the friction. It is proved that  linear  critical values of a   spectral problem are
 nonlinear thresholds  leading to the onset of secondary  steady-state flows, the nonlinear phenomenon observed in laboratory experiments.

\end{abstract}

\maketitle


\section{Introduction}
To study the inverse energy cascade towards large scales \cite{K} of plane flows, Sommeria and Verron \cite{Sommeria86,Sommeria,Sommeria87}  presented magnetohydrodynamic experiments by using electronically driven flows in a closed box, containing a thin horizontal layer of liquid metal. The box is bottomed with electromagnets producing a uniform vertical magnetic field. The flow velocity is small so that the upper free surface is negligible. The three-dimensional motion  reduces a two-dimensional one as the vertical movement in the thin horizontal layer fluid    can be ignored. The energy dissipation of the fluid motion counts for viscosity and the Hartmann layer friction applied on the bottom of liquid metal.

The non-dimensional governing   equations of the  two-dimensional approximation motion for the velocity $\v$ and pressure $p$ in the domain $(0,1)\times (0,1)$ are  \cite{Sommeria86,Sommeria,Sommeria87}
\bal
\frac{\p \v}{\p t} + \v\cdot \nabla \v + \nabla p - \frac1{Re}\Delta \v +\frac{\v}{Rh} = \f,\,\,\, \nabla \cdot \v=0.\label{aa1}
\eal
Here $Re$ is the Reynolds number, $Rh$ is the Rayleigh number measuring the Hartmann bottom  friction and   $\f$ is the Lorentz  driving force defined by electric currents so that
\be \nabla\times \f= \frac{\pi^2}2 \sin(2\pi x)\sin(2\pi y)\,\,\, \mbox{ and } \,\,\,\int^1_0\int^1_0 |\nabla\times \f|dxdy =2.\ee
The stream function $\psi$ and the vorticity of the fluid motion are defined as
\be \Big(\frac{ \partial \psi}{\partial y},\,-\frac{\partial \psi}{\partial x}\Big)= \v,\,\,\,\omega =\nabla\times \v= -\Delta \psi.\ee
The   vorticity  formulation of (\ref{aa1}) is
\bbe
-\frac{\p\Delta \psi}{\p t} + J(\psi,\Delta \psi)-\frac{\Delta \psi}{Rh}  + \frac{\Delta^2 \psi}{Re}=\frac{\pi^2}2 \sin(2\pi x)\sin(2\pi y),
\label{new1}
\bee
where the nonlinear convective term is written as the Jacobian
\be J(\psi,\Delta\psi)=\p_x\psi \,\p_y\Delta\psi -\p_y\psi\p_x\Delta\psi.
\ee

The basic flow of (\ref{new1}) is dependant  on the parameters $Re$ and $Rh$. It is convenient  to use the modified system \cite{Thess} of (\ref{new1})  expressed through
\bbe
-\frac{ \p\Delta \psi}{\p t}+J(\psi,\Delta\psi)+(-\mu\Delta +\nu\Delta^2)(\psi-\sin x\sin y)=0.\label{NS1}
\bee
In reference to  \cite{Sommeria86,Sommeria,Sommeria87,Thess}, the stream function is assumed to satisfy   the free slip boundary condition
\bbe\label{bdc} \psi|_{\p\Omega}=\Delta\psi|_{\p\Omega}=0
\bee
for  the modified fluid domain $\Omega = (0,2\pi )\times (0,2\pi ),$ and  is demonstrated  in  the Fourier expansion
\bbe \psi= \sum_{n,m\ge 1} a_{n,m} \sin \frac{nx}{2}\sin \frac{my}2.\bee
The parameters $\nu$ and $\mu$ are defined by the  transformations \cite{Chen2019}
\bbe  \label{ReRh}  
Rh= \frac{2\sqrt{\mu+2\nu}}{\mu \pi}\,\,\mbox{ and }\,\, Re=\frac{8\pi\sqrt{\mu+2\nu}}{\nu}.
\bee
 With this modification, we have the basic steady-state flow
\be \psi_0 = \sin x\sin y. \ee

The basic flow exhibits four vortices in Figure \ref{ff1}. The experiments \cite{Sommeria86,Sommeria,Sommeria87} show transitions of the basic flow in a scenario of  inverse energy cascade towards to large scales. The principal transition amongst them is the steady-state bifurcation of $\psi_0$ into a secondary flow,  which is  sketched in Figure \ref{ff1} (b).

 To the understanding of the transition, Thess \cite{Thess} demonstrated  critical stability parameters $(\nu_c,\mu_c)$ of a spectral problem linearized from (\ref{NS1})-(\ref{bdc}) so that linear stable and unstable domains are defined. The author \cite{Chen2019} provided   nonlinear stability  analysis of a vortex flow and employed a  numerical spectral scheme to study all possible  linear spectral solutions  together with secondary flows as a result of nonlinear saturation of primary linear instability.    The vortex instability   with respect to   two vortex merging phenomena was also discussed by Meunier  {\it et al.} \cite{2005merging}
and Cerretelli and  Williamson\cite{2003merging}.
The experimental studies \cite{Sommeria86,Sommeria} of the non-parallel flow $\sin x\sin y$ are  developed from the magnetohydrodynamic experiment of Bondarenko {\it et al.} \cite{Bon1979} on the steady-state bifurcation of the parallel Kolmogorov flow $\sin x$. The existence of  secondary steady-state flows and secondary temporal  periodic flows bifurcating from the Kolmogorov flow   has been studied extensively \cite{Chen2002,Chen2005,Yud}.

However, the basic flow $\psi_0$ is non-parallel and  rigorous instability analysis for the  secondary flow existence  is missing. In the study of the linear spectral problem, Thess \cite{Thess} emphasized that
 linear stability theory
is not able to predict the structure of flows above the instability
threshold, but it is a matter of bifurcation theory to decide whether stationary secondary solutions exist at
all.
He also suggested the linear spectral study to  be continued in the following   two directions:
(i) the formation of a secondary flow as a result of
nonlinear saturation of the primary instability and (ii) linear stability analysis of the secondary flow (see  Orszag and Patera \cite{Or} on this stability problem  for  a parallel  flow).
 In the present study, we   solve  problem (i)      by showing the formation of a secondary steady-state flow  resulting from nonlinear saturation of the linear instability.
For problem (ii), no existing rigorous analysis is available.  The linear stability analysis of    secondary flow   \cite{Or} on  the parallel Poiseuille flow problem is not applicable to (ii).  To sketch   stability of the secondary steady-state  flow, we use numerical computation via  a finite difference scheme. Selected numerical results  show  that  the secondary  flow, close to its threshold and   observed in the laboratory experiments  \cite{Sommeria86,Sommeria,Sommeria87}, can be obtained  by taking an initial state in its vicinity.  The stability of the secondary flow and the  positive viscosity $\nu>0$ ensure the convergence of the numerical  scheme so that the secondary  flow attracts the non-stationary flow  starting from the initial state.
\begin{figure}
 \centering
\includegraphics[height=.45\textwidth, width=.45\textwidth]{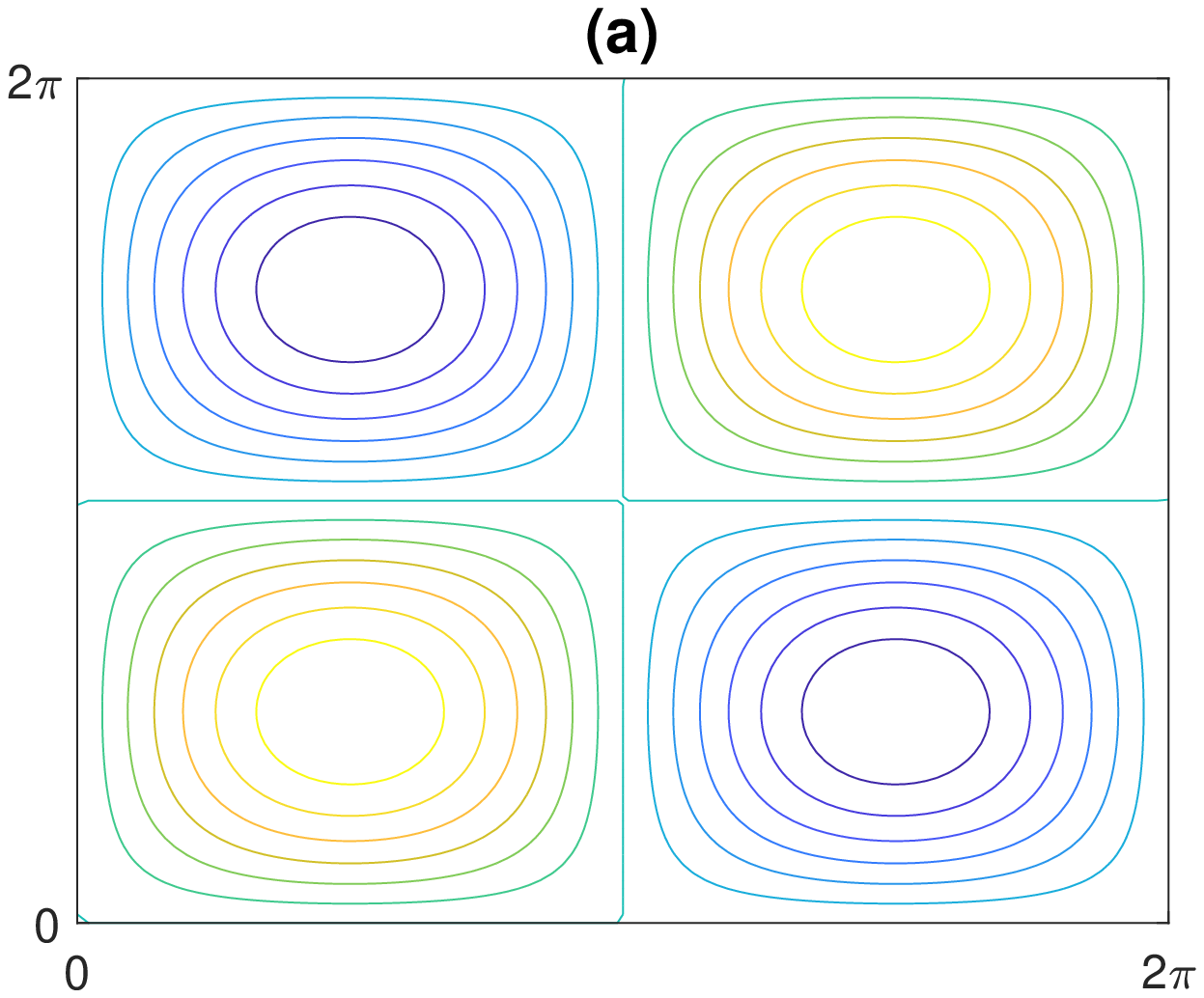}
\includegraphics[height=.45\textwidth, width=.45\textwidth]{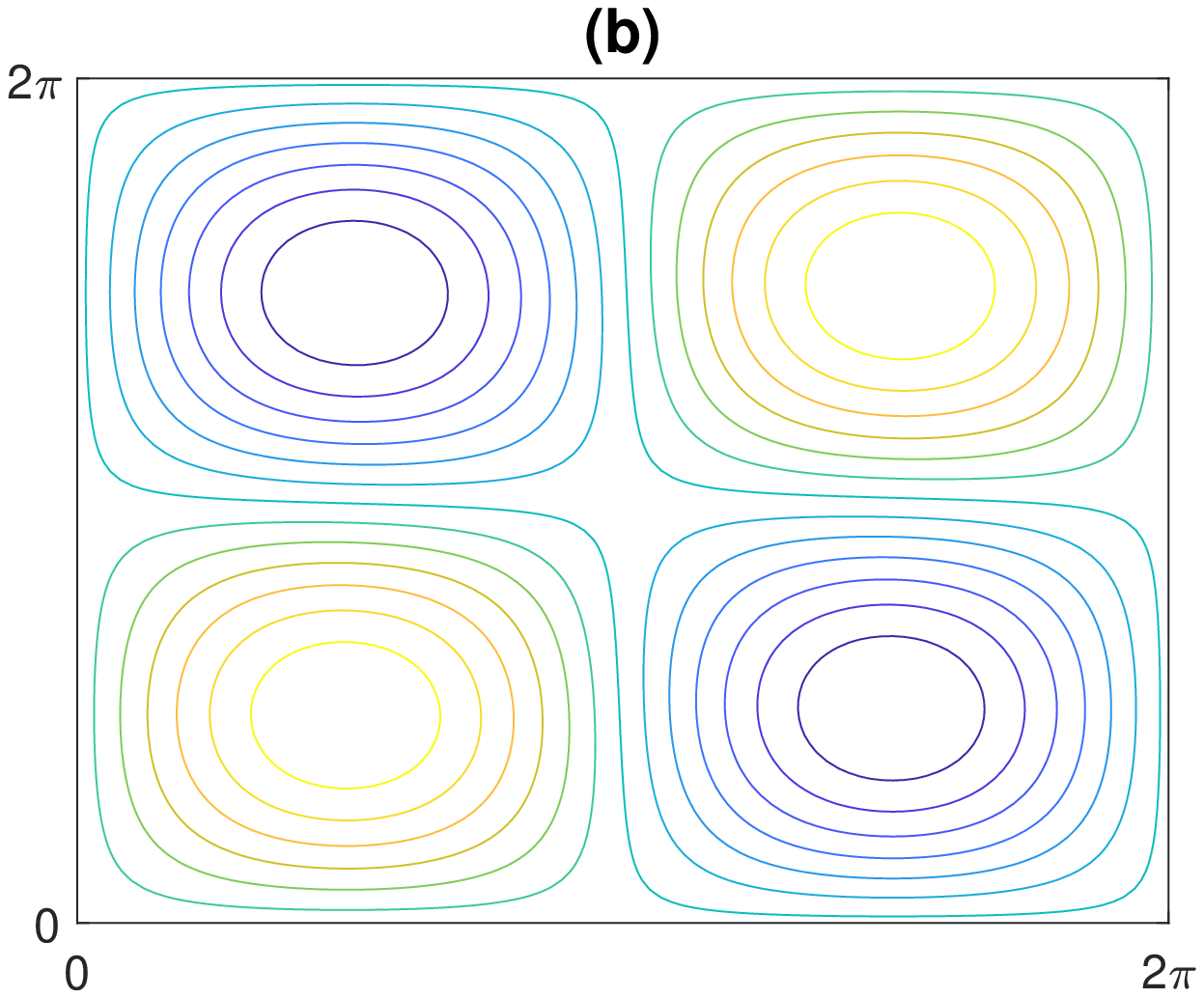}

 \caption{(a) The basic steady-state flow $\psi_0$;  (b) the contour lines of the function $\psi_0-0.1 \sin \frac x2\sin \frac y2$, representing a profile of the secondary flow in the magnetohydrodynamic  experiment \cite{Sommeria}.}
  \label{ff1}
 \end{figure}

It is the principal  purpose of present paper to show the existence of the secondary flows, which are to be contained in the Hilbert space
\be  H^4 &=& \Bigg\{\psi =\sum_{n,m\ge 1} a_{n,m}\sin \frac{nx}2\sin\frac{my}2 \Big |\,\,
\\
&&\|\psi\|_{H^4} = \Bigg(\sum_{n,m\ge 1} \Big(1+(\frac{n^2+m^2}4)^2\Big)^2 |a_{n,m}|^2\Bigg)^\frac12<\infty  \Bigg\}.
\ee
The secondary flows will be constructed by nonlinear perturbation of the basic flow $\psi_0$ by using the eigenfunctions of the spectral problem
\bbe \lambda \Delta\psi=-\mu  \Delta \psi+ \nu\Delta^2  \psi+ J(\psi_0,(2+\Delta)\psi),\,\,\, 0\ne\psi= \sum_{n,m\ge 1} a_{n,m}\sin \frac {nx}2\sin \frac{my}2,\label{LL}
\bee
which is linearized from   (\ref{NS1}). The eigenfunctions will be studied in the following three linear orthogonal subspaces of $H^4$:
\bbe \label{o}E_1&=&\Bigg\{ \psi\in H^4 | \,\,\, \psi = \sum_{n,m\ge 1; \, n,m \mbox{ \footnotesize odd}}a_{n,m}\sin \frac {nx}2\sin\frac {my}2\Bigg\},
\\
\label{o1} E_2&=&\Bigg\{ \psi\in H^4 | \,\,\, \psi = \sum_{n,m\ge 1; \, n \mbox{ \footnotesize odd};\, m \mbox{ \footnotesize even}}a_{n,m}\sin \frac {nx}2\sin\frac {my}2\Bigg\},
\\
\label{o2}E_3&=&\Bigg\{ \psi\in H^4 | \,\,\, \psi = \sum_{n,m\ge 1; \, n\mbox{ \footnotesize even};\, m \mbox{ \footnotesize odd}}a_{n,m}\sin \frac {nx}2\sin\frac {my}2\Bigg\}.
\bee
The algebraic form of  spectral equation (\ref{LL}) displayed in next section shows that  the mode
$\sin \frac{nx}2\sin \frac{my}2$ is influenced by the four modes $\sin \frac{(n\pm 2)x}2 \sin \frac{(m\pm2)y}2$ in $x$ and $y$ directions.  The introduction of these orthogonal subspaces   implies  the validity of  the invariance property of (\ref{LL})   in the following sense:
\be  \sin \frac{(n\pm 2)x}2 \sin \frac{(m\pm2)y}2\in E_i\,\,\mbox{ whenever }\,\, \sin \frac {nx}2\sin\frac {my}2 \in E_i \,\,\,\,\, i=1, 2, 3.
\ee

The main result of the present paper  reads as:
\begin{theorem}\label{main} 
(i).
 Let $\nu>0$, $\mu\ge 0$ and $\lambda +\mu +\frac12 \nu >0$. Then  spectral problem (\ref{LL}) has at most three linear independent eigenfunctions. These eigenfunctions are contained in the set
$E_1\cup E_2\cup E_3$.

 (ii) Assume that  spectral problem (\ref{LL}) admits a critical solution $(\lambda,\psi, \nu,\mu)=(0,\psi_c,\nu_c,\mu_c)$ for $\nu_c>0$,  $\mu_c\ge 0$ and    $\psi_c\in E_{i }$ for an integer $1\le i \le 3$.
 Then   there exist a function
$\psi_{i }\in H^4$ and a real $\delta$ so that  system (\ref{NS1})-(\ref{bdc}) has a steady-state solution $(\psi,\nu,\mu)$ branching  off the bifurcation point $(\psi_0,\nu_c,\mu_c)$ in the direction of
$\psi_c$:
\bbe &&\psi=\psi_0+\epsilon \psi_{c} +\epsilon^2 \psi_{i },\,\,\,\nu=\nu_c+\epsilon\delta \nu_c,\,\,\mu=\mu_c +\epsilon\delta \mu_c,\label{dd1}
\label{dd1x}
\bee
provided that $\epsilon>0$ is sufficiently small. Here $\sigma$ and $\psi_i$ are uniformity bounded  functions of $\epsilon$ for small $\epsilon$, and $\psi_i$ is in the orthogonal  complement of the eigenfunction space $\mbox{\rm span}\{\psi_c\}$ or  $\psi_i\in H^4/ \mbox{\rm span}\{\psi_c\}$.

\end{theorem}

\begin{remark}
This theorem shows the secondary flow bifurcating in the direction of $\psi_c$. If $\psi_c$ is replaced by the eigenfunction $-\psi_c$, we have another secondary flow bifurcating in the direction of $-\psi_c$.
\end{remark}

This paper is structured  as follows.   The spectral analysis for the theoretical base of Theorem \ref{main}  is established in Section 2, which contains the proof of Theorem \ref{main} (i). The second assertion  of this theorem  is  proven in Section 3  by developing a bifurcation technique of Rabinowitz \cite{Rab} on a B\'enard problem.  Theorem \ref{main} requires the existence of a linear critical spectral solution $(\nu_c,\mu_c,\psi_c)$.  Section 4 contains a discussion of the existence  and  connection  of Theorem \ref{main} with  Crandall- Rabinowitz bifurcation theorem.
To enrich the theoretical result, we display numerical spectral solutions and use a finite difference scheme to locate a secondary flow in accordance with the experimental observation of \cite{Sommeria,Sommeria87} in Section 5. This numerical study aids  the stability detection of the secondary flow.     Moreover, in addition to  the initial stage of    inverse energy cascade    profiled in Figure \ref{ff1}, a larger scale topological transition   via three vortices into two are presented in Section 5.

\section{Linear spectral analysis}
We begin  with the spectral  assertion of  Theorem \ref{main}.

\subsection{Proof of Theorem \ref{main} (i)}
\begin{proof} Let $(\cdot,\cdot)$ denote the inner product of real $L_2$ as
\be (\varphi,\phi) = \frac1{\pi^2}\int^{2\pi}_0\int^{2\pi}_0 \varphi \phi dx dy.
\ee
 Taking the $L_2$ inner product of the spectral equation (\ref{LL}) with  $-(\Delta +2)\psi$ and employing  integration by parts, we have
\bbe
0 &=& (-\lambda \Delta -\mu\Delta \psi+\nu \Delta^2\psi+J(\psi_0,(\Delta+2)\psi),(-\Delta-2)\psi)\nonumber
\\
&=&(-\lambda \Delta  -\mu\Delta \psi+\nu \Delta^2\psi,(-\Delta-2)\psi).\label{aa20}
\bee
This together with (\ref{LL}) becomes
\bbe
\label{a00}0=\sum_{ n,m\ge 1} \beta_{n,m}(\lambda +\mu +\nu\beta_{n,m})(\beta_{n,m}-2)|a_{n,m}|^2\,\,\,\mbox{ for }\,\,\, \beta_{n,m} = \frac14(n^2+m^2), 
\bee
or
\bbe
\lefteqn{\sum_{ n,m\ge 1;\,\beta_{n,m}>2} \beta_{n,m}(\lambda+\mu +\nu\beta_{n,m})(\beta_{n,m}-2)|a_{n,m}|^2}\nonumber
\\
&=&
 \beta_{1,1}(\lambda+\mu +\nu\beta_{1,1})(2-\beta_{1,1})|a_{1,1}|^2+\beta_{1,2}(\lambda+\mu +\nu\beta_{1,2})(2-\beta_{1,2})|a_{1,2}|^2\nonumber
 \\
 &&+\beta_{2,1}(\lambda+\mu +\nu\beta_{2,1})(2-\beta_{2,1})|a_{2,1}|^2.\label{eigg}
\bee
By the observation  $2>\beta_{1,2}=\beta_{2,1} > \beta_{1,1}=\frac12$ and the condition $\lambda +\mu +\frac12 \nu >0$, we see that
 \be \lambda +\mu+\nu \beta_{n,m} \ge  \lambda +\mu+\frac12\nu>0 \mbox{ for }\,\,\, n,\,m \ge 1\ee
 and  the both sides of (\ref{eigg}) are non-negative. Hence  the right-hand side of (\ref{eigg}) is positive due to the non-zero eigenfunction property $\psi\ne 0$.  Therefore we may firstly assume the  term involving $a_{1,1}$ on the right-hand side  of (\ref{eigg}) being    positive  or $a_{1,1}\ne 0$,

On the other hand, let $a_{n,m}=0$ whenever $n\le 0$ or $m\le 0$. Spectral problem (\ref{LL}) is  formulated  as \cite{Chen2019}
\bbe\label{mmm1}
\lefteqn{\sum_{n,m\ge 1}\beta_{n,m} (\lambda+\mu +\nu\beta_{n,m})  a_{n,m} \sin \frac{nx}{2} \sin \frac{my}{2}}
\\ \nonumber
&=& -\sum_{n,m\ge-2}^\infty\Big\{\frac{n-m}{8}[(\beta_{n-2,m-2}-2)a_{n-2,m-2}
-(\beta_{n+2,m+2}-2)a_{n+2,m+2}]\nonumber
\\
&& +\frac{n+m}{8}[(\beta_{n-2,m+2}-2)a_{n-2,m+2}
- (\beta_{n+2,m-2}-2)a_{n+2,m-2} ]\Big\}\sin \frac{nx}{2}\sin \frac{my}{2}.\hspace{3mm}\nonumber
\bee
This implies that the non-zero coefficient $a_{1,1}$ produces   the coefficients $a_{n,m}$ for odd integers $n,m\ge 1$. That is,   the eigenfunction $\psi \in E_1$  is generated by the mode $\sin \frac x2\sin \frac y2$.
 Moreover, the derivation of (\ref{eigg}) implies that
\bbe
\lefteqn{\beta_{1,1}(\lambda+\mu +\nu\beta_{1,1})(2-\beta_{1,1})|a_{1,1}|^2}\nonumber
\\ &=&\sum_{n,m\geq 1; \, n,m \mbox{ \footnotesize odd};\,\beta_{n,m}>2} \beta_{n,m}(\lambda+\mu +\nu\beta_{n,m})(\beta_{n,m}-2)|a_{n,m}|^2.\label{aaa2}
\bee

Similarly,    we may suppose $\psi \in E_2$ when $a_{1,2}\ne 0$ and $\psi \in E_3$ when $a_{2,1}\ne 0$.
 Additionally,  the corresponding  coefficients are subject to the equations
\bbe
 \lefteqn{\beta_{1,2}(\lambda+\mu +\nu\beta_{1,2})(2-\beta_{1,2})|a_{1,2}|^2}\nonumber
\\ &=&\sum_{n,m\ge 1; \, n \mbox{ \footnotesize odd}; \, m \mbox{ \footnotesize even}; \,\beta_{n,m}>2} \beta_{n,m}(\lambda+\mu +\nu\beta_{n,m})(\beta_{n,m}-2)|a_{n,m}|^2\,\label{aaa22}
\bee
for $a_{1,2}\ne 0$, and
\bbe \lefteqn{\beta_{2,1}(\lambda+\mu +\nu\beta_{2,1})(2-\beta_{2,1})|a_{2,1}|^2}\nonumber
\\ &=&\sum_{ n,m\ge 1; \, n \mbox{ \footnotesize even}; \, m \mbox{ \footnotesize odd};\,\beta_{n,m}>2} \beta_{n,m}(\lambda+\mu +\nu\beta_{n,m})(\beta_{n,m}-2)|a_{n,m}|^2\label{aaa222}
\bee
for $a_{2,1}\ne 0$.
The  proof of Assertion (i) is complete.
\end{proof}

\subsection{Spectral simplicity property}
To construct the secondary flows, we have to use the eigenfunction simplicity  property shown in the following result.

\begin{theorem} \label{th2}   Let  $\nu>0$, $\mu\ge 0$ and $\lambda+\mu +\frac12 \nu >0$.   Then we have  eigenfunction space dimension estimate:
\bbe \label{one}\dim \Big\{ \psi \in E_i  | \,\,\, \lambda \Delta \psi=-\mu\Delta \psi+\nu \Delta^2\psi+J(\psi_0,(\Delta+2)\psi)\Big\}\leq 1, \,\,\,i=1,2,3.
\bee
Moreover, if $(\lambda,\psi,\nu,\mu)$ is a spectral solution of (\ref{LL}) for $\psi \in E_1\cup E_2\cup E_3$, then we have
\bbe( (-\lambda\Delta-\mu\Delta +\nu \Delta^2)\psi, \psi^*)\ne 0.\label{two}
\bee
Here $\psi^*$ is the conjugate eigenfunction of $\psi$ subject to the conjugate equation of (\ref{LL}):
\bbe \lambda \Delta \psi^*&=&-\mu  \Delta \psi^*+ \nu\Delta^2  \psi^*+ (-\Delta-2)J(\psi_0,\psi^*),\label{mmb}
\\
  \psi^*&=&\sum_{n,m\ge 1}^\infty a^*_{n,m}\sin\frac{nx}{2}\sin \frac{my}{2}\ne 0,\nonumber
\bee
produced by employing the  $L_2$  pairing $(\cdot,\cdot)$.
\end{theorem}

\begin{proof}

To  show  the validity of  (\ref{one}),   we suppose that there is  a spectral solution $(\lambda,\psi,\nu,\mu)$ with the eigenfunction $\psi\in E_1$. To the contrary, if (\ref{one}) with $i=1$ is not true,  there exists an additional spectral solution $(\lambda,\hat\psi,\nu,\mu)$ with  the eigenfunction $\hat \psi \in E_1$  linearly independent of $\psi$ and involving  expansion coefficients $\hat a_{n,m}$.
 It follows from (\ref{aaa2}) that  $\hat a_{1,1}\ne 0$. Therefore, we have the additional spectral solution $(\lambda, \psi-\frac{ a_{1,1}}{\hat a_{1,1} }\hat \psi,\nu,\mu)$.
Using the eigenfunction $\psi-\frac{ a_{1,1}}{\hat a_{1,1} }\hat \psi$ instead of $\psi$ in  (\ref{aaa2}), we have
 \bbe
0=\sum_{ n,m \mbox{ \footnotesize odd};\,\beta_{n,m}>2} \beta_{n,m}(\lambda+\mu+\nu\beta_{n,m})(\beta_{n,m}-2)|a_{n,m}-\frac{a_{1,1}}{\hat a_{1,1}}\hat a_{n,m}|^2,\label{a11}
\bee
which together with the condition $\lambda +\mu +\nu \frac12>0$  gives
\be a_{n,m}=\frac{a_{1,1}}{\hat a_{1,1}}\hat a_{n,m}\,\,\,\mbox{ or }\,\,\,  \psi= \frac{a_{1,1}}{\hat a_{1,1}}\hat\psi_{1,1}.
\ee
Hence $\hat \psi$ and $\psi$ are linearly dependent. This leads to a contraction and thus (\ref{one}) holds true for $i=1$.

 Arguing in the same way, we obtain (\ref{one}) for $i=2$ and $3$.

To verify (\ref{two}), we first assume the eigenfunction $\psi\in E_1$. Consider the  conjugate spectral problem (\ref{mmb}), which can be formulated  in the algebraic equation
\bal\nonumber
0=&\sum_{n,m\ge -2; \, n,m \mbox{ \footnotesize odd}}\Bigg\{ \beta_{n,m} (\lambda+\mu+\nu\beta_{n,m})  a^*_{n,m}
\\ \nonumber
&-(\beta_{n,m}-2)\bigg\{\frac{n-m}{8}(a^*_{n-2,m-2}-a^*_{n+2,m+2})\nonumber
+\frac{n+m}{8}(a^*_{n-2,m+2} - a^*_{n+2,m-2} )\bigg\}\Bigg\}\sin \frac{nx}{2}\sin \frac{my}{2}.
\eal
Here  $a^*_{n,m}=0$ whenever  $n \le 0$ or $m\le 0$. The previous equation is rewritten as
\bbe
0&=&\sum_{n,m\ge -2; \, n,m \mbox{ \footnotesize odd} }\Bigg\{ [(\lambda+\mu)\beta_{n,m} +\nu\beta_{n,m}^2]  \frac{a^*_{n,m}}{\beta_{n,m}-2}  \label{mmm22x}
\\ \nonumber
&&-\frac{n\!-\!m}{8}(a^*_{n-2,m-2}\!-\! a^*_{n+2,m+2}) 
 +\frac{n\!+\!m}{8}(a^*_{n-2,m+2} \!-\! a^*_{n+2,m-2} )\Bigg\}\sin \frac{nx}{2}\sin \frac{my}{2}.
\bee
Moreover, for  $a'_{m,n}=\frac{a^*_{n,m}}{\beta_{n,m}-2}$, equation (\ref{mmm22x}) becomes
\be\nonumber
0&=&\sum_{n,m\ge -2; \, n, m \mbox{ \footnotesize odd} }\Bigg\{ [(\lambda+\mu)\beta_{m,n} +\nu\beta_{m,n}^2]  a'_{m,n} 
\\ \nonumber
&& +\frac{n-m}{8}[(\beta_{m-2,n-2}-2)a'_{m-2,n-2}
-(\beta_{m+2,n+2}-2)a'_{m+2,n+2}]\nonumber
\\
&& +\frac{n+m}{8}[(\beta_{m+2,n-2}-2)a'_{m+2,n-2}
- (\beta_{m-2,n+2}-2)a'_{m-2,n+2} ]\Bigg\}\sin \frac{nx}{2}\sin \frac{my}{2}.
\ee
Replacing the indices   $n, \,m$ by $\tilde m,\, \tilde n $ respectively, and then omitting the index superscript tildes,   we have
\bbe\nonumber
0&=&\sum_{n,m\ge -2; \, n,m \mbox{ \footnotesize odd} }\Bigg\{ [(\lambda+\mu)\beta_{n,m} +\nu\beta_{n,m}^2]  a'_{n,m} 
\\ \nonumber
&& +\frac{n-m}{8}[(\beta_{n-2,m-2}-2)a'_{n-2,m-2}
-(\beta_{n+2,m+2}-2)a'_{n+2,m+2}]\nonumber
\\
&& +\frac{n+m}{8}[(\beta_{n-2,m+2}-2)a'_{n-2,m+2}
- (\beta_{n+2,m-2}-2)a'_{n+2,m-2} ]\Bigg\}\sin \frac{mx}{2}\sin \frac{ny}{2}.\hspace{3mm}\label{mmm111x}
\bee
The algebraic equation system  for the coefficients of (\ref{mmm111x})  is  identical to that for the coefficients   of (\ref{mmm1}), when the supercritical primes are omitted. Thus by (\ref{one}),  we have the relationship between the expansion coefficients of the eigenfunction $\psi$ and those of its conjugate counterpart $\psi^*$:
\be a^*_{n,m} =(\beta_{m,n}-2)a_{m,n}.\ee
Hence, we have
\bbe(\Delta (-\lambda -\mu +\nu \Delta)\psi, \psi^*)= \sum_{n,m\ge 1; \, n,m \mbox{ \footnotesize odd}} \beta_{n,m}(\lambda+\mu+\nu\beta_{n,m})(\beta_{n,m}-2) a_{n,m} a_{m,n}.\label{aaa1}
\bee

  Note that $\psi\ne 0$ implies  $a_{1,1}\ne 0$ due to (\ref{aaa2}). Moreover, from the algebraic equation defined by the first component of  (\ref{mmm1}) with respect to the mode $\sin \frac x2\sin\frac y2$, it follows that the coefficient $a_{1,1}$ is  proportional to $a_{1,3}-a_{3,1}$. Hence $a_{3,1}\ne a_{1,3}$ due to $a_{1,1}\ne 0$. Therefore  the  Cauchy inequality
\bbe 2a_{1,3}a_{3,1} < a_{1,3}^2+a_{3,1}^2\label{y2}\bee
holds true.
It follows  from  (\ref{aaa1}), (\ref{y2}) and the Cauchy inequality $2a_{n,m}a_{m,n} \le a_{n,m}^2+ a_{m,n}^2$  that
\be (\Delta (-\lambda -\mu +\nu \Delta)\psi, \psi^*)
&=&-\beta_{1,1}(\lambda+\mu+\nu\beta_{1,1})(2-\beta_{1,1})a_{1,1}^2
\\&& +\!\!\!\!\!\!\!\!\!\!\!\!\sum_{n, m\ge 1; \, n,m \mbox{ \footnotesize odd};\,\beta_{n,m}>2} \!\!\!\!\!\!\!\!\!\!\!\!\beta_{n,m}(\lambda+\mu+\nu\beta_{n,m})(\beta_{n,m}-2) a_{n,m} a_{m,n}
\\
&<&-\beta_{1,1}(\lambda+\mu+\nu\beta_{1,1})(2-\beta_{1,1})a_{1,1}^2
\\&&+\sum_{n, m\ge 1; \, n,m \mbox{ \footnotesize odd};\,\beta_{n,m}>2} \!\!\!\!\!\!\!\!\!\!\!\!\beta_{n,m}(\lambda+\mu+\nu\beta_{n,m})(\beta_{n,m}-2) a_{n,m}^2,
\ee
which equals zero due to (\ref{aaa2}). This gives the validity of (\ref{two}) when $\psi\in E_1$.

Moreover, if (\ref{LL}) has a  spectral solution $(\lambda,\psi,\nu,\mu)$ with the eigenfunction
\bbe \psi=\sum_{n,m\ge 1;\,n \mbox{ \footnotesize odd}; \, m \mbox{ \footnotesize even}}a_{n,m}\sin\frac{nx}2\sin\frac{my}2\in E_2,\label{E2}
\bee
   the algebraic equation (\ref{mmm1})  for the coefficients $a_{n,m}$ of  $\psi$ reduces to
\bbe
\lefteqn{0=  \sum_{n,m\ge -2;\,n \mbox{ \footnotesize odd}; \, m \mbox{ \footnotesize even}}\Bigg\{\beta_{n,m}(\lambda+\mu+\nu \beta_{n,m})a_{n,m}}\nonumber
\\
&&+\frac{n-m}{8}[(\beta_{n-2,m-2}-2)a_{n-2,m-2}-(\beta_{n+2,m+2}-2)a_{n+2,m+2}] \nonumber
\\
&&+\frac{n+m}{8}[(\beta_{n-2,m+2}-2)a_{n-2,m+2}-(\beta_{n+2,m-2}-2)a_{n+2,m-2} ]\Bigg\}\sin \frac{nx}{2}\sin \frac{my}{2}.\hspace{6mm}\label{mmm10}
\bee
 Therefore, the  conjugate spectral problem (\ref{mmb}) has a solution $(\lambda,\psi^*,\nu,\mu)$, subject to the algebraic equation
\bbe
0&=&\sum_{n,m\ge -2; \, n \mbox{ \footnotesize odd}; \, m \mbox{ \footnotesize even} }\Bigg\{ [(\lambda+\mu)\beta_{n,m} +\nu\beta_{n,m}^2]  \frac{a^*_{n,m}}{\beta_{n,m}-2}  \label{mmm22}
\\ \nonumber
&&-\frac{n\!-\!m}{8}(a^*_{n-2,m-2}\!-\! a^*_{n+2,m+2}) 
 +\frac{n\!+\!m}{8}(a^*_{n-2,m+2} \!-\! a^*_{n+2,m-2} )\Bigg\}\sin \frac{nx}{2}\sin \frac{my}{2}.
\bee
Here the assumption $a^*_{n,m}=0$ whenever  $n \le 0$ or $m\le 0$ is used.
For  $a'_{m,n}=\frac{a^*_{n,m}}{\beta_{n,m}-2}$, equation (\ref{mmm22}) becomes
\begin{align}\nonumber
0&=\sum_{n,m\ge -2; \, n \mbox{ \footnotesize even}; \, m \mbox{ \footnotesize odd} }\Bigg\{ [(\lambda+\mu)\beta_{n,m} +\nu\beta_{n,m}^2]  a'_{n,m} 
\\ \nonumber
& +\frac{n-m}{8}[(\beta_{n-2,m-2}-2)a'_{n-2,m-2}
-(\beta_{n+2,m+2}-2)a'_{n+2,m+2}]\nonumber
\\
& +\frac{n+m}{8}[(\beta_{n-2,m+2}-2)a'_{n-2,m+2}
- (\beta_{n+2,m-2}-2)a'_{n+2,m-2} ]\Bigg\}\sin \frac{mx}{2}\sin \frac{ny}{2}.\label{mmm111}
\end{align}
This together with (\ref{one}) implies that spectral problem (\ref{LL}) has  a spectral solution $(\lambda,\tilde  \psi,\nu,\mu)$ with
\bbe \label{hat}\tilde  \psi=\sum_{n,m\ge 1; \, n \mbox{ \footnotesize even}; \, m \mbox{ \footnotesize odd} }a_{n,m} \sin \frac{nx}2\sin\frac{my}2\in E_3,
\bee
subject to the equation
\begin{align}\nonumber
0&=\sum_{n,m\ge -2; \, n \mbox{ \footnotesize even}; \, m \mbox{ \footnotesize odd} }\Bigg\{ [(\lambda+\mu)\beta_{n,m} +\nu\beta_{n,m}^2]  a_{n,m} 
\\ \nonumber
& +\frac{n-m}{8}[(\beta_{n-2,m-2}-2)a_{n-2,m-2}
-(\beta_{n+2,m+2}-2)a_{n+2,m+2}]\nonumber
\\
& +\frac{n+m}{8}[(\beta_{n-2,m+2}-2)a_{n-2,m+2}
- (\beta_{n+2,m-2}-2)a_{n+2,m-2} ]\Bigg\}\sin \frac{mx}{2}\sin \frac{ny}{2}.\label{xxx1}
\end{align}
Therefore, we have
\bbe a^*_{n,m} = (\beta_{n,m}-2)a_{m,n}.\label{xxx3}\bee

We show that $a_{n,m}\not\equiv a_{m,n}$. Otherwise, if $a_{n,m}\equiv a_{m,n}$, equation (\ref{xxx1}) becomes
\begin{align}\nonumber
0&=\sum_{n,m\ge -2; \, n \mbox{ \footnotesize odd}; \, m \mbox{ \footnotesize even} }\Bigg\{ [(\lambda+\mu)\beta_{n,m} +\nu\beta_{n,m}^2]  a_{n,m} 
\\ \nonumber
& -\frac{n-m}{8}[(\beta_{n-2,m-2}-2)a_{n-2,m-2}
-(\beta_{n+2,m+2}-2)a_{n+2,m+2}]\nonumber
\\
& -\frac{n+m}{8}[(\beta_{n-2,m+2}-2)a_{n-2,m+2}
- (\beta_{n+2,m-2}-2)a_{n+2,m-2} ]\Bigg\}\sin \frac{mx}{2}\sin \frac{ny}{2}.\label{xxx2}
\end{align}
Adding (\ref{xxx2}) to (\ref{mmm10}), we have
\be\sum_{n,m\ge 1; \, n \mbox{ \footnotesize odd}; \, m \mbox{ \footnotesize even} }[(\lambda+\mu)\beta_{n,m} +\nu\beta_{n,m}^2]  a_{n,m}\sin \frac{mx}{2}\sin \frac{ny}{2}=0\ee
or
$a_{n,m}\equiv 0$. This leads to a contradiction.  Hence
\bbe \label{yy3}a_{n,m}\not\equiv a_{m,n}.\bee

By  (\ref{xxx3}), we have
\bbe(\Delta (-\lambda -\mu +\nu \Delta)\psi, \psi^*)= \!\!\!\!\!\!\!\!\!\sum_{n,m\ge 1; \, n \mbox{ \footnotesize odd}; \, m \mbox{ \footnotesize even}} \!\!\!\!\!\!\!\!\!\beta_{n,m}(\lambda+\mu+\nu\beta_{n,m})(\beta_{n,m}-2) a_{n,m} a_{m,n}\label{aaa11}
\bee

If $a_{1,2}=a_{2,1}$, we use (\ref{aa20}), (\ref{yy3}) and  Cauchy inequality to  obtain from (\ref{aaa11}) that
\be
\lefteqn{(\Delta (-\lambda -\mu +\nu \Delta)\psi, \psi^*)}
\\
&<& \beta_{1,2}(\lambda+\mu+\nu\beta_{1,2})(\beta_{1,2}-2) a_{1,2}^2
\\
&&+\sum_{n,m\ge 1; \, n \mbox{ \footnotesize odd}; \, m \mbox{ \footnotesize even};\,\beta_{n,m}>2}\beta_{n,m}(\lambda+\mu+\nu\beta_{n,m})(\beta_{n,m}-2) \frac{a_{n,m}^2 +a_{m,n}^2}2 \nonumber\label{aaa110}
\\
&=& \beta_{1,2}(\lambda+\mu+\nu\beta_{1,2})(\beta_{1,2}-2) \frac{a_{1,2}^2+a_{2,1}^2}2
\\
&&+\frac12\sum_{n,m\ge 1; \, n \mbox{ \footnotesize odd}; \, m \mbox{ \footnotesize even};\,\beta_{n,m}>2}\beta_{n,m}(\lambda+\mu+\nu\beta_{n,m})(\beta_{n,m}-2) a_{n,m}^2  \nonumber\label{aaa110}
\\
&&+\frac12\sum_{n,m\ge 1; \, n \mbox{ \footnotesize even}; \, m \mbox{ \footnotesize odd};\,\beta_{n,m}>2}\beta_{n,m}(\lambda+\mu+\nu\beta_{n,m})(\beta_{n,m}-2) a_{n,m}^2
\ee
which equals zero due to (\ref{aaa22}) and (\ref{aaa222}). This gives  (\ref{two}) under the condition $a_{1,2}=a_{2,1}$.

Actually, we can assume  $a_{1,2}=a_{2,1}=1$, since  the spectral problem is linear. This gives (\ref{two}) for the eigenfunction $\psi\in E_2$.
This derivation also implies the validity of (\ref{two}) when  the eigenfunction $\psi \in E_3$.
The proof of Theorem \ref{th2} is completed.

\end{proof}

\def\Ran{\mathrm Ran}
\def\Dom{\mathrm Dom}

\def\L{\mathcal L}
\section{Proof of  Theorem \ref{main} (ii)}

\begin{proof} {\sl Firstly,} we introduce a flow invariant  space so that Fredholm alternative theory can be applied for the critical eigenfunction  $\psi_c\in E_i$.
Theorem \ref{th2} shows that $\psi_c \in E_{i }$ is a simple eigenfunction. That is,
\be \dim\Bigg \{ \psi\in E_{i }\, \Big|\, 0=-\mu_c \Delta\psi+\nu_c \Delta^2\psi + J(\psi_0, (\Delta+2)\psi)\Bigg\} =1
\ee
and (\ref{two}) holds true.
$E_{i}$ is an invariant space  of linear spectral problem (\ref{LL}) but is not flow invariant for  nonlinear problem (\ref{NS1})-(\ref{bdc}).
The linear perturbation  part of the bifurcating solution is expected to be in the  eigenfunction space  span$\{\psi_c\}\subset E_i$. Therefore, we need to consider the bifurcation in a nonlinear  flow invariant space of (\ref{NS1})-(\ref{bdc}) and the space is generated nonlinearly from the linear space $E_{i }$.

To do so, we use the summation notation
\be
\sum_{1}&=&\sum_{1\le n,\,m\, \,{\sl odd}}+ \sum_{2\le n,\,m \,\,{\sl even}},
\\
\sum_2&=&\sum_{1\le n\,\, {\sl odd};\,2\le m \,\,{\sl even}}+ \sum_{2\le n,\, m \,\, {\sl even}},
\\
\sum_3&=&\sum_{2\le n\, \,{\sl even};\,1\le m \,\,{\sl odd}}+ \sum_{2\le n,\, m \,\, {\sl even}}.
\ee
We define  the $H^4$ subspaces
\be H^4_{i}=\Bigg\{ \psi \in H^4 \,\Big|\,\, \psi =\sum_{i} a_{n,m} \sin \frac{nx}2\sin \frac{my}2
\Bigg\} \,\,\, \mbox{ for }\,\,\, i=1,2,3,
\ee
and the $L_2$ subspaces
\be H_{i} =\Bigg\{ \psi =\sum_{i} a_{n,m} \sin \frac{nx}2\sin \frac{my}2\Big| \,\,\, \|\psi\|_{L_2} =\Big( \sum_i a_{n,m}^2\Big)^\frac12 <\infty \Bigg\} \,\,\, \mbox{ for }\,\,\, i=1,2,3.\ee

This  definition  ensures  $H_i^4 \supset E_i$ for $i=1, 2, 3$  and $H_i^4$ is orthogonal to $E_j$ if $i\ne j$.  Hence the assertion of Theorem \ref{th2} remains valid when $E_{i }$ is replaced by $H^4_{i }$. That is, the eigenfunction simplicity property holds true in $H^4_{i }$. The nonlinear  flow invariant property of $H^4_{i }$ is valid  in the following sense
\bbe \label{inv}\Delta^{-2}J(\varphi,\Delta\phi) \in H^4_{i }  \mbox{ whenever $\varphi,\,\phi \in H^4_{i }$}.\bee
This invariance property  is confirmed by the estimate
\be \|\Delta ^{-2}J(\varphi,\Delta \phi)\|_{H^4}&\le&  \|J(\varphi,\Delta \phi)\|_{L_2}
 \\
 &\le& \|\nabla \varphi \|_{L_4} \| \nabla \Delta\phi\|_{L_4} \le C \| \varphi \|_{H^4} \| \phi \|_{H^4},
\ee
due to H\"older inequality and Sobolev imbedding, and the multiplication computation
\be J(\varphi, \Delta \phi) &=&\Bigg( \sum_{i} a_{n,m} \frac n2\cos\frac{nx}2\sin\frac{my}2\Bigg)\Bigg(\sum_{i} b_{n,m} \frac {m}2\beta_{n,m}\sin\frac{nx}2\cos\frac{my}2\Bigg)
\\
&&- \Bigg(\sum_{i} a_{n,m} \frac m2\sin\frac{nx}2\cos\frac{my}2\Bigg)\Bigg(\sum_{i} b_{n,m} \frac {n}2\beta_{n,m}\cos\frac{nx}2\sin\frac{my}2\Bigg)
\\
&=& \sum_i c_{n,m} \sin \frac {nx}2 \sin \frac{my}2
\ee
for   coefficients $c_{n,m}$ rearranged from $a_{n,m}$ and $b_{n,m}$. Here we use the  functions
\be \varphi =\sum_i a_{n,m} \sin \frac{nx}2\sin \frac{my}2 \in H^4_i\,\,\mbox{ and }\,\, \phi = \sum_i b_{n,m} \sin \frac {nx}2 \sin \frac{my}2\in H^4_i.
\ee

Now we rewrite the critical spectral problem as
\be \L \psi_{c}= 0 \,\,\,\mbox{ for }\,\,\, \L\psi= -\mu_c \Delta\psi +\nu_c \Delta^2 \psi+ J(\psi_0, (\Delta+2)\psi) .\ee
We see that $\L$ maps $H^4_{i }$ into $H_{i }$.
 To employ the Fredholm theory, we define the range of $\L$ as
\be \Ran(\L)=\Big\{ \varphi \in H_{i }\Big|\,\, \mbox{ there exists $\phi\in H^4_{i }$ so that } \L\phi =\varphi\Big\}.
\ee
 It readily seen that $\Ran(\L )$ is the space orthogonal to $\psi_c^*$, the conjugate eigenfunction of $\psi_{c}$, in the following sense:
\be \Ran(\L ) = \Big \{ \psi \in H_{i }\Big|\, \, (\psi,\psi_{i }^*)=0\Big\}.\ee
By the Fredholm alternative theory of Laplacian operators, $\L$ has an inverse operator
\bbe \L ^{-1}: \Ran(\L ) \mapsto H_{i }^4\label{c7}
\bee
 so that
\bbe \label{L}\|\L ^{-1}\psi\|_{H^4_{i }} \leq C_1 \| \psi \|_{L_2},\,\,\psi \in \Ran(\L )\bee
for a constant $C_1$.

{\sl Secondly,}  following Rabinowitz \cite{Rab} on a B\'enard problem, we seek the  secondary   steady-state solution $(\psi,\nu,\mu)$  branching from the bifurcation point $(\psi_0,\nu_{c},\mu_{c})$ in the direction of $\psi_{c}$ as
\bbe \psi=\psi_0 + \epsilon \psi_{c} + \epsilon^2 \psi_{i },\,\,\, \nu =\nu_{c} + \epsilon \sigma \nu_{c}, \,\,\mu =\mu_{c} + \epsilon \sigma \mu_{c}\label{d1}
\bee
 for a function $\psi_{i }\in H^4_{i }$ and a  real  $\sigma$, provided that $\epsilon>0$ is sufficiently small.

Substitution of the predicted solution (\ref{d1}) into the stationary form of  (\ref{NS1}),  or the equation
\begin{align*} 0
=&(-\mu  \Delta+ \nu\Delta^2 )( \psi-\psi_0)+ J(\psi_0,(2+\Delta)(\psi-\psi_0))+J(\psi-\psi_0,\Delta(\psi-\psi_0))
\\
=&(-(\mu -\mu_c) \Delta+ (\nu-\nu_c)\Delta^2 )( \psi-\psi_0)+ \L(\psi-\psi_0)+J(\psi-\psi_0,\Delta(\psi-\psi_0)),
\end{align*}
produces the equation
\be 
0&=& ( -\epsilon \sigma \mu_{c}\Delta+ \epsilon \sigma \nu_{c}\Delta^2) (\epsilon \psi_{c} + \epsilon^2 \psi_{i })+\L  (\epsilon \psi_{c} + \epsilon^2 \psi_{i })
\\
&&+J( \epsilon \psi_{c} + \epsilon^2 \psi_{i },\Delta (\epsilon \psi_{c} + \epsilon^2 \psi_{i })).
\ee
Since $\L \psi_{c}=0$, the previous equation  can be rewritten as
 \bbe\label{x2}
 \sigma(-\mu_{c}\Delta +\nu_{c}\Delta^2) \psi_{c} +\L  \psi_{i }=F_\epsilon(\sigma,\psi_{i })
 \bee
with
\be F_\epsilon(\sigma,\psi_{i })
&=&- \epsilon \sigma(-\mu_{c}\Delta +\nu_{c}\Delta^2)  \psi_{i }-J(  \psi_{c} + \epsilon \psi_{i },\Delta ( \psi_{c} + \epsilon\psi_{i })).
\ee

To show the existence of  the unknowns $\psi_{i }$ and $\sigma$, we take $L_2$ inner product of (\ref{x2}) with $\psi_c^*$ to obtain
\bbe\label{x22}
 \sigma((-\mu_{c}\Delta +\nu_{c}\Delta^2) \psi_{c},\psi_c^*) +(\L  \psi_{i },\psi_c^*)=(F_\epsilon(\sigma,\psi_{i }),\psi_c^*).
 \bee
Applying Theorem \ref{th2} with $E_{i }$ replaced by $H^4_{i }$, the invariance property (\ref{inv}) and the identity
\be (\L \psi_{i },\psi_{c}^*)=(\psi_{i },\L ^*\psi_{c}^*)=0,
\ee
we may rewrite (\ref{x22}) as
\bbe \label{x5}\sigma = \frac{(F_\epsilon(\sigma,\psi_{i }),\psi^*_{c})}{((-\mu_{c}\Delta +\nu_{c}\Delta^2)\psi_{c},\psi_{c}^*)}.
\bee
The combination of (\ref{x2}) and (\ref{x5}) yields
\bbe \label{x4x}
\L\psi_{i }=  F_\epsilon(\sigma,\psi_{i })-  \frac{(-\mu_{c}\Delta +\nu_{c}\Delta^2)\psi_{c}(F_\epsilon(\sigma,\psi_{i }),\psi^*_{c})}{((-\mu_{c}\Delta +\nu_{c}\Delta^2)\psi_{c},\psi_{c}^*)}.
\bee
The nonlinear invariance property (\ref{inv}) implies  $F_\epsilon(\sigma,\psi_{i })\in H_{i }$.  It is readily seen that the right-hand side of (\ref{x4x}) is in $\Ran(\L)$.
Therefore, we may use the inverse of $\L$ to produce
\bbe \label{x4}
\psi_{i }=\L ^{-1} \Big( F_\epsilon(\sigma,\psi_{i })-  \frac{(-\mu_{c}\Delta +\nu_{c}\Delta^2)\psi_{c}(F_\epsilon(\sigma,\psi_{i }),\psi^*_{c})}{((-\mu_{c}\Delta +\nu_{c}\Delta^2)\psi_{c},\psi_{c}^*)}
\Big).
\bee
For simplicity of notation, we rewrite the equations  (\ref{x5}) and (\ref{x4}) in the following form
\bbe (\sigma,\psi_{i })= G_\epsilon(\sigma,\psi_{i }),
\bee
where the two components of the operator $G_\epsilon(\sigma,\psi_{i })$ represent respectively the right-hand sides of (\ref{x5}) and (\ref{x4}).
Thus, to seek the solution $(\psi,\nu,\mu)$ in (\ref{d1}) becomes to confirm the existence of  the fixed point for the operator $G_\epsilon$.

{\sl Finally,}  it remains to prove that $G_\epsilon$ is a contraction operator  mapping a complete metric space into itself.
The  complete matric space is defined as
\be X = \Bigg\{ (\sigma,\psi)\in (-\infty,\infty)\times H^4_{i }\Big|\,\, \,\, \|(\sigma,\psi)\|_X= |\sigma| + \|\psi\|_{H^4_{i }}\leq C\Bigg\}.
\ee
 Here $C>0$ is a constant to be defined afterward.

To show the contraction property, we use the boundedness of $\L ^{-1}$ in (\ref{L}), the expressions  (\ref{x5}) and (\ref{x4}), and  H\"older inequality to produce
\be
\|G_\epsilon(\sigma,\psi_{i })\|_X
&\leq&\Bigg(\frac{\|\psi^*_{c}\|_{L_2}}{|((-\mu_{c}\Delta +\nu_{c}\Delta^2)\psi_{c},\psi_{c}^*)|}
\\&&+ C_1  \Big(1+ \frac{\|(-\mu_{c}\Delta +\nu_{c}\Delta^2)\psi_{c}\|_{L_2}\|\psi^*_{c}\|_{L_2}}{|((-\mu_{c}\Delta +\nu_{c}\Delta^2)\psi_{c},\psi_{c}^*)|}\Big)\Bigg)\|F_\epsilon(\sigma,\psi_{i })\|_{L_2}.
\ee
This yields, by renaming  the constant bounded by  the large brackets  in the right-hand side of the previous equation  as $C_2$,
\be
\|G_\epsilon(\sigma,\psi_{i })\|_X
&\leq& C_2\|F_\epsilon(\sigma,\psi_{i })\|_{L_2}. 
\ee
Hence, by   H\"older inequality and  Sobolev imbedding inequality, we have
\be
\|G_\epsilon(\sigma,\psi_{i })\|_X&\le & C_2\|J(  \psi_{c} + \epsilon \psi_{i },\Delta ( \psi_{c} + \epsilon\psi_{i }))
+\epsilon \sigma(-\mu_{c}\Delta +\nu_{c}\Delta^2)  \psi_{i }\|_{L_2}
\\
&\le & C_2\Big(\|\nabla(  \psi_{c} + \epsilon \psi_{i })\|_{L_4} \|\nabla \Delta ( \psi_{c} + \epsilon\psi_{i })\|_{L_4}
\\&&+\epsilon \sigma\|(-\mu_{c}\Delta +\nu_{c}\Delta^2)  \psi_{i }\|_{L_2}\Big)
\\
&\leq & C_3\Big( \|\psi_{c}\|_{H^4_{i }}^2+ 2\epsilon \|\psi_{c}\|_{H^4_{i }}\|\psi_{i }\|_{H^4_{i }}+ \epsilon^2\|\psi_{i }\|_{H^4_{i }}^2+ \epsilon \sigma\|\psi_{i }\|_{H^4_{i }}\Big)
\\
&\leq & C_4\Big( 1+ \epsilon  C+ \epsilon^2C^2+\epsilon C^2\Big)
\ee
for  the  constants  $C_k$ independent of $(\sigma,\psi_i)\in X$ and $\epsilon>0$.
Therefore, we obtain
\bbe\label{in}
\|G_\epsilon(\sigma,\psi_{i })\|_X\le C\,\,\mbox{ for  $(\sigma,\psi_{i })\in X$},
\bee
provided that
\be
\frac{C}2 = C_4
\ee
and
\be
C_4(\epsilon  C+ \epsilon^2C^2+\epsilon C^2)= C_4(\epsilon  + 2\epsilon^2C_4+2\epsilon C_4) C \le \frac12 C,
\ee
by taking  $\epsilon>0$ sufficiently small.
The property (\ref{in}) implies the  injection property $ G_\epsilon: X \mapsto X.$

Arguing in the same manner, we have the  contraction property:
\be
\|G_\epsilon(\sigma,\psi_{i })-G_\epsilon(\sigma',\psi_{i }')\|_X\leq \frac12 \|(\sigma,\psi_{i })-(\sigma',\psi_{i }')\|_X
\ee
for $ (\sigma,\psi_{i }), \,(\sigma',\psi_{i }')\in X$, provided that $\epsilon>0$  sufficiently small.
Therefore, by the Banach contraction mapping principle, the operator $G_\epsilon$ with small $\epsilon>0$ admits a unique fixed point $(\sigma,\psi_{i })\in X$.
This confirms the existence of the steady-state solution $(\psi, \mu,\nu)$  of (\ref{NS1}) and (\ref{bdc}) in the form of (\ref{d1}).

The uniform boundedness of the $\sigma$ and $\psi_i$ with respect to $\epsilon$  is given by (\ref{in}). The property $\psi_i \in H^4_i/\mbox{\rm span}\{\psi_c\}$ is implied from (\ref{x4}) due to the Fredholm operator property  $\mathcal{L}^{-1}: \Ran(\mathcal{L}) \mapsto H_i^4/\mbox{\rm span}\{\psi_c\}$.
The proof of Theorem \ref{main} is completed.

\end{proof}
\section{Discussion   on Theorem \ref{main}}

\subsection{Discussion on  the existence of critical spectral solution}

From viewpoint of numerical computation, a  critical spectral solution   $(\nu_c,\mu_c,\psi_c)$  of   spectral problem (\ref{LL}) with $\lambda=0$ can be  calculated by using  algebraic equation (\ref{mmm1}). However, from viewpoint of rigorous analysis, its  existence remains unsolved.
In the unidirectional Kolmogorov flow problem,  its linear spectral equation can be transformed into a continuous fraction equation, from which the  existence of critical value is obtained. The coefficients of the corresponding eigenfunction can expressed as multiplications of continuous fractions ( see \cite{Chen2002,Chen2005,Yud}) for both real and non-real eigenfunctions.   However, this continuous fraction approach is no longer applicable to the present multi-directional flow problem. To facilitate the understanding of the difficulty, we  discuss the problem  with the aid of  a truncated form.

For simplicity, we only consider  eigenfunction in $E_1$.
Linear spectral problem  (\ref{LL}) or (\ref{mmm1}) with $\lambda=0$  can be rewritten  in the form
\bbe
0&=&\sum_{-2 \le n,m; n+m \le k;\, n,m \,\mbox{\footnotesize odd}}\Big\{\alpha_{n,m} b_{n,m}
 +(n-m)[b_{n-2,m-2}
-b_{n+2,m+2}]\nonumber
\\
&& +(n+m)[b_{n-2,m+2}
- b_{n+2,m-2} ]\Big\}\sin \frac{nx}{2}\sin \frac{my}{2}\label{mmm11}
\bee
as $k\to \infty$. Here
\be\alpha_{n,m} = 8 \frac{\beta_{n,m}(\mu + \nu \beta_{n,m})}{\beta_{n,m}-2}, \,\, b_{n,m} =(\beta_{n,m}-2) a_{n,m} \,\mbox{ for }\, n,m \ge 1, \,\mbox{ otherwise } \, b_{n,m} =0. \ee
If $A_k$ represents the square matrix of influence coefficients and $B_k$  the truncated eigenvector consists of $b_{n,m}$, then (\ref{mmm11}) becomes
\be A_k B_k=0\ee
For example, for the approximating spectral problem for $k=8$, we  have the 10-dimensional truncation equation
\be
&&0=\alpha_{1,1}b_{1,1}-2b_{1,3}+2b_{3,1}
\\
 && 0=  \alpha_{1,3} b_{1,3}-4b_{1,5}+2b_{3,5}+2b_{1,1}-4b_{3,1}
\\
&& 0=   \alpha_{3,1} b_{3,1}+4b_{5,1}-2b_{5,3}-2b_{1,1}+4 b_{1,3}
\\
&& 0=  \alpha_{1,5} b_{1,5}-6b_{1,7}+4b_{1,3}-6b_{3,3}
\\
&& 0=   \alpha_{5,1} b_{5,1}+6b_{7,1}-4 b_{3,1}+6 b_{3,3}
\\
&& 0= \alpha_{3,3} b_{3,3}+6 b_{1,5}-6 b_{5,1}
\\
&& 0=  \alpha_{1,7} b_{1,7}+6b_{1,5}-8b_{3,5}
\\
&& 0=   \alpha_{7,1} b_{7,1}-6b_{5,1}+8 b_{5,3}
\\
&& 0= \alpha_{3,5} b_{3,5}-2b_{1,3}+8 b_{1,7}-8 b_{5,3}
\\
&& 0= \alpha_{5,3} b_{5,3}+2b_{3,1}+8b_{3,5}- 8b_{7,1}
\ee
or
{\footnotesize \bbe
0=A_8 B_8=\left(\begin{array}{cccccccccccccc}
\alpha_{1,1}&-2&2&0& 0& 0    &0&0& 0&0
\\
2 & \alpha_{1,3}& -4& -4& 0&0   &0&0& 2&0
\\
-2& 4& \alpha_{3,1}&0&4&0      &0&0& 0&-2
\\
0 &4& 0&\alpha_{1,5}&0&-6     &-6&0& 0&0
\\
0& 0&-4&0&\alpha_{5,1}&6  &0&6& 0&0
\\
0&0&0&6&-6&\alpha_{3,3}   &0&0& 0&0
\\
0&0&0&6&0&0  &\alpha_{1,7}&0& -8&0
\\
0&0&0&0&-6&0 &0&\alpha_{7,1}& 0&8
\\
0&-2&0&0&0&0 &8&0& \alpha_{3,5}&-8
\\
0&0&2&0&0&0 &0&-8& 8&\alpha_{5,3}
\end{array}\right)\left( \begin{array}{c} b_{1,1}\\
 b_{1,3}\\ b_{3,1} \\ b_{1,5} \\ b_{5,1} \\ b_{3,3}\\b_{1,7}\\
 b_{7,1}\\ b_{3,5} \\ b_{5,3}
 \end{array} \right).\label{A8}
 \bee
 }
 The truncation seems quit harsh, but the solution to this equation is reasonably  close to that of the original one. The existence of the corresponding spectral critical solution becomes the existence
 of  the  critical value $(\nu,\mu)=(\nu_c,\mu_c)\ne (0,0)$ satisfying  the determinant equation
\be \det(A_8(\nu,\mu))=0.\ee
Since $\alpha_{1,1}<0$ and $\alpha_{n,m}>0$ for $(n,m)\ne (1,1)$, we see that $\lim_{\nu+\mu\to \infty}\det(A_8(\nu,\mu))\to -\infty$. Therefore, the existence of the root $(\nu_c,\mu_c)$ can be confirmed if one finds some $(\nu,\mu)$ so that  $\det(A_8(\nu,\mu))>0$. However, it is laborious to prove this positivity property although $A_k$ is always skew-symmetric when $\nu=\mu=0$.
\begin{figure}
 \centering
 \includegraphics[height=.26\textwidth, width=.45\textwidth]{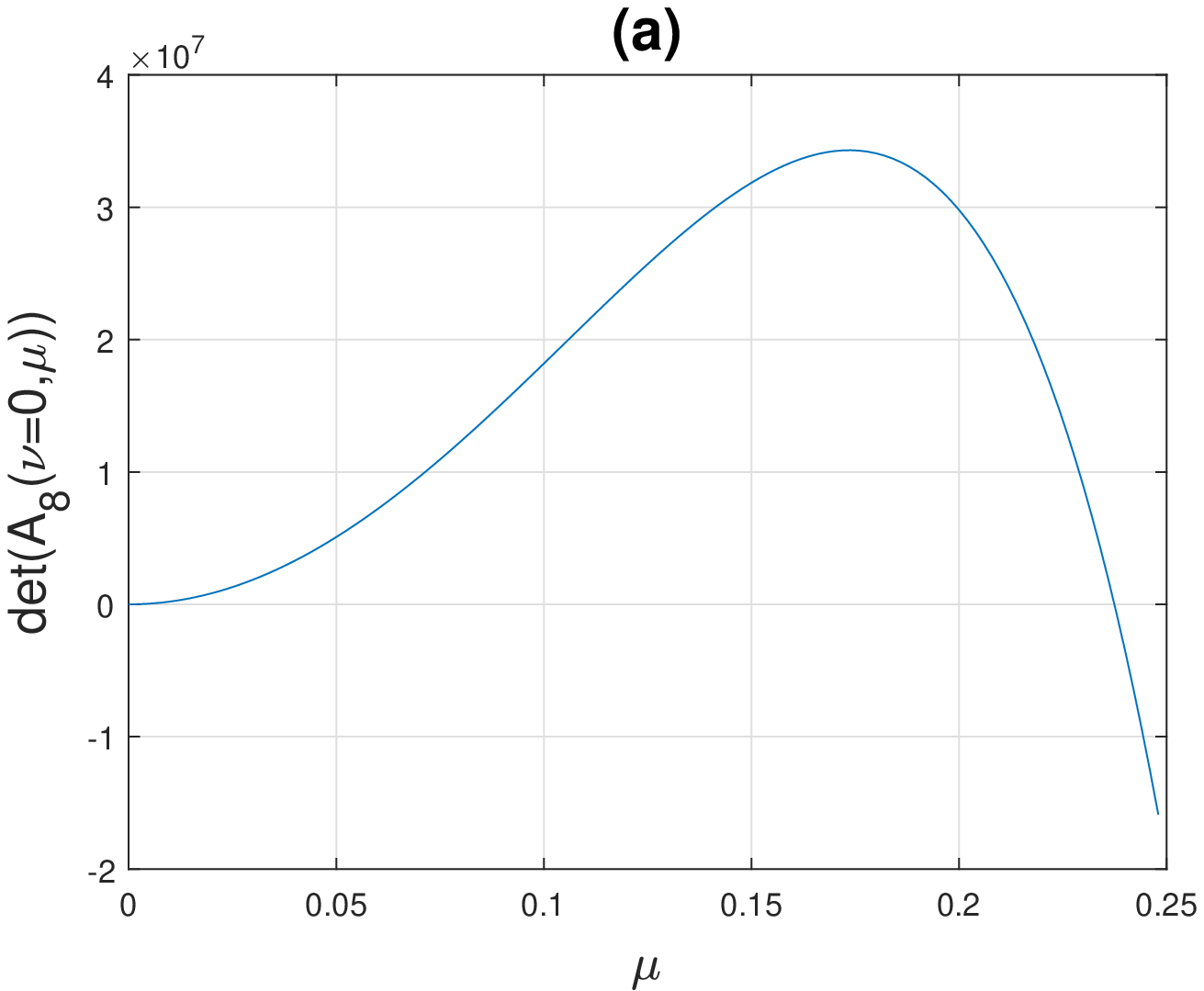}
\includegraphics[height=.26\textwidth, width=.45\textwidth]{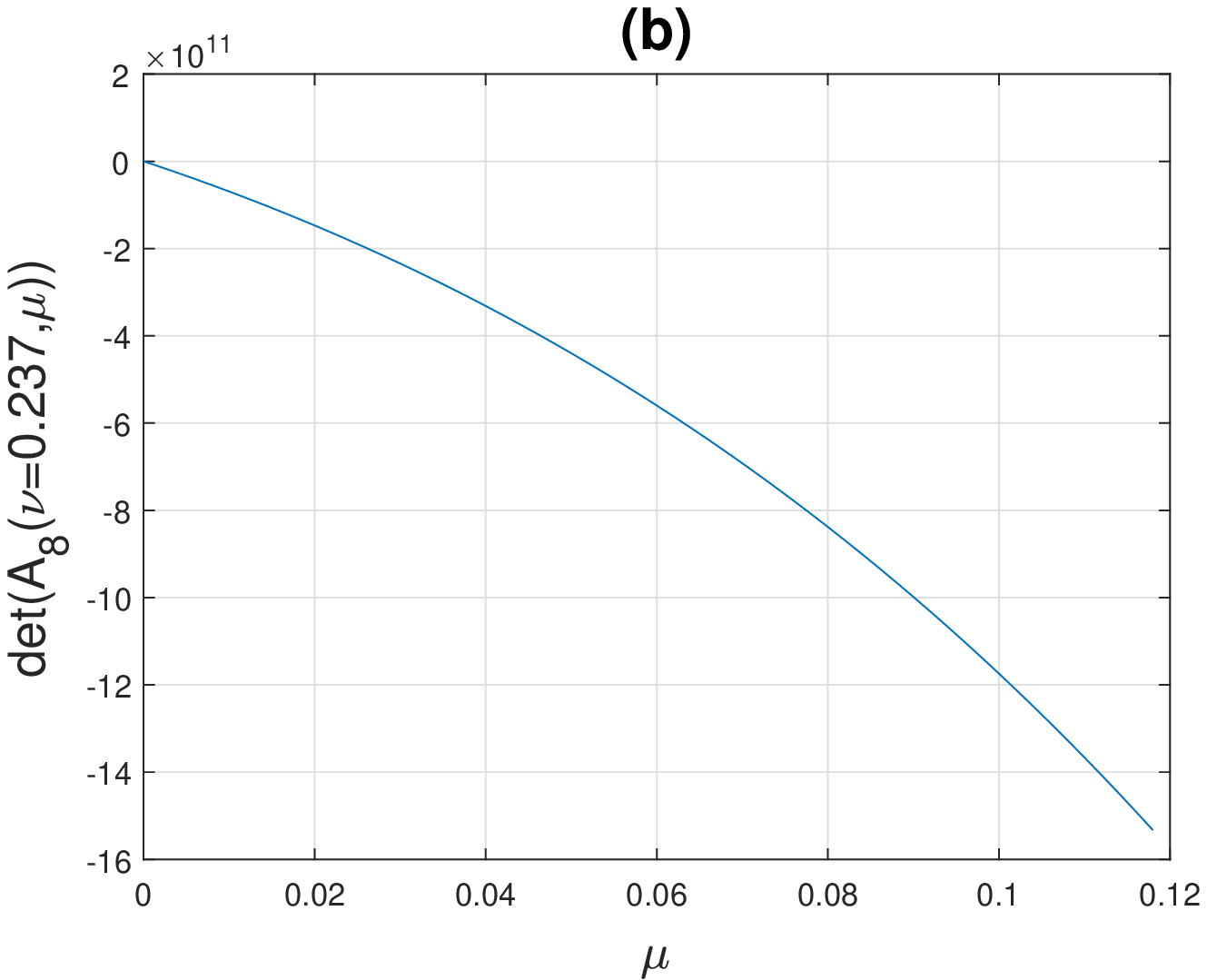}
\includegraphics[height=.26\textwidth, width=.45\textwidth]{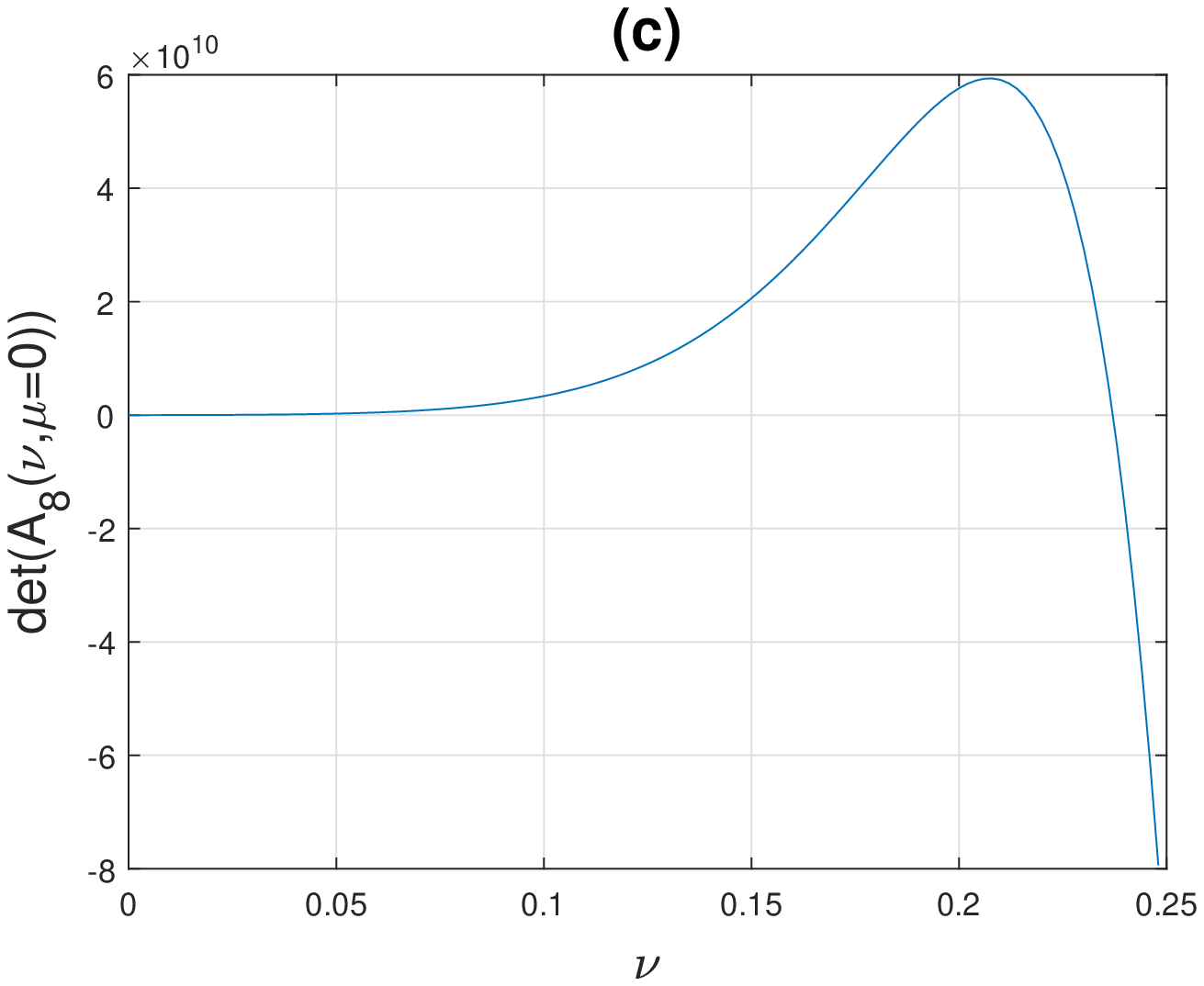}
\includegraphics[height=.26\textwidth, width=.45\textwidth]{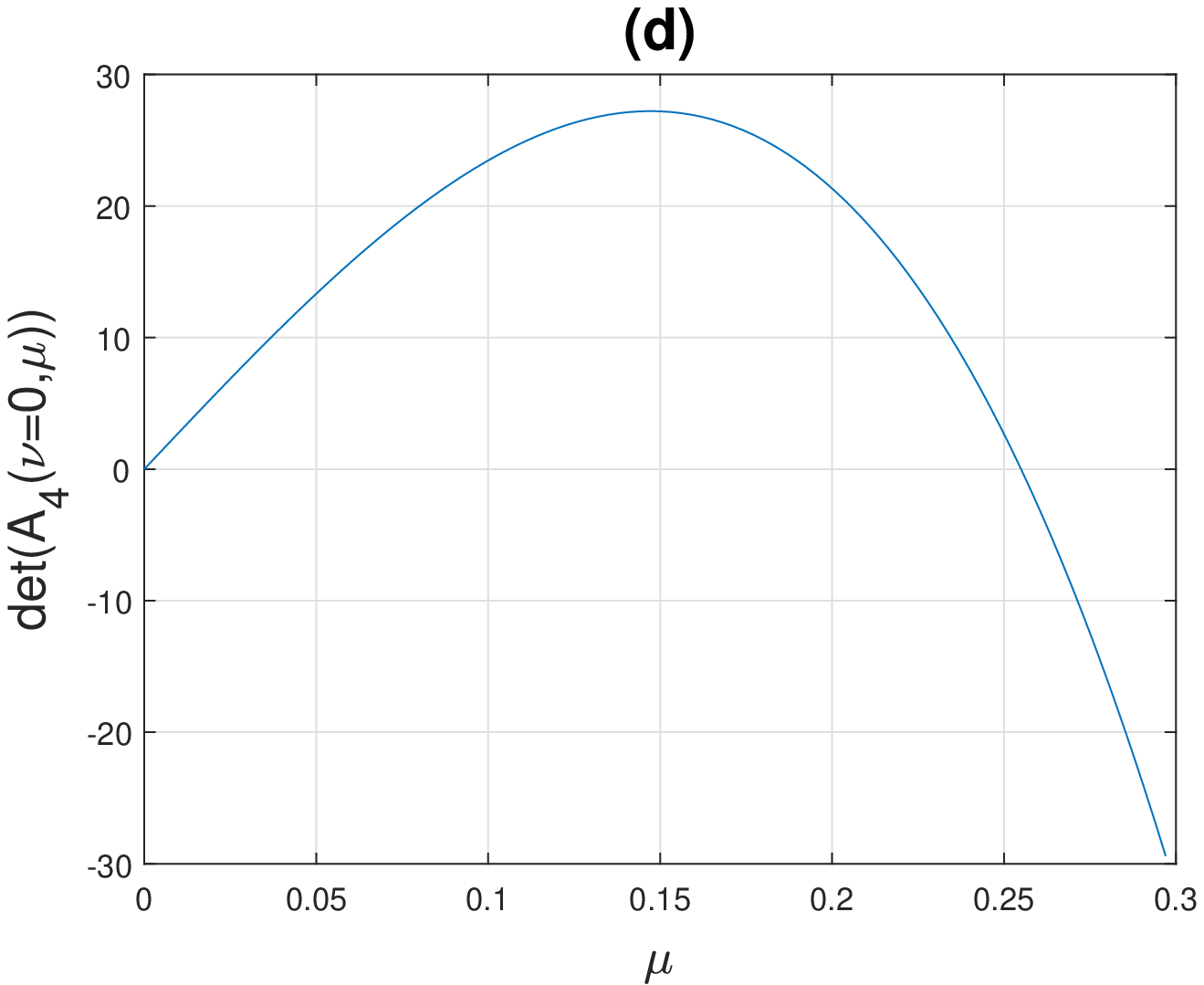}
 \caption{Selected computation results for the polynomials $\det(A_8(\nu,\mu))$ and $\det(A_4(\nu,\mu))$ crossing horizontal zero lines.}\label{ff2}
\end{figure}
 To show the positivity property and then the root existence, we use  the selected  computation result in  Figure \ref{ff2}, which shows the roots $(\nu_c,\mu_c)\approx (0,0.236)$ from (a) and  $(\nu_c,\mu_c)\approx (0.237,0)$ from (b)-(c). This  result is   comparable  with that  computed by Thess \cite{Thess}.

 Similar behaviour exists for  the polynomial det$(A_k)$ as $k$ increases. Actually, the behavior of the polynomial $\det(A_8)$ such as  transversal crossing the horizontal zero line is comparable  to that of  the third degree polynomial (see Figure \ref{ff2} (d))
\be \det(A_4)= \alpha_{1,1} \alpha_{1,3} \alpha_{3,1} + 16 \alpha_{1,1} + 4 \alpha_{1,3}+4 \alpha_{3,1},\ee although the deviation  with respect to the  root increases. Here  $A_4$  is the $3\times 3$ matrix in the top-left corner of $A_8$.

\subsection{Connection to Crandall- Rabinowitz bifurcation theorem}

The bifurcation result can be obtained  by Crandall- Rabinowitz bifurcation theorem \cite{CR}, although the secondary flow shown in  Theorem \ref{main} is more informative due to the construction by  Banach fixed point theorem. As an alternative way, we would like to show the existence of  the secondary flow as a consequence of the following.
\begin{theorem}\label{tho} (Crandall and Rabinowitz \cite[Theorem 1.7]{CR}) Let $X$, $Y$ be Banach spaces, $V$ a neighborhood of $0$ in $X$ and
\bbe F: (-1,1)\times V \mapsto Y \label{In}\bee
have the properties
\begin{description}
\item[(a)] $F(\tau,0)=0$ for $ |\tau|<1$,
\item[(b)] The partial derivatives $F_\tau$, $F_\psi$ and $F_{\tau\psi}$ exist and are continuous,
\item[(c)] The kernel space $N(F_\psi(0,0))$ and the orthogonal compliment $Y/\Ran(F_\psi(0,0))$ are one-dimensional,
\item[(d)] $F_{\tau\psi}(0,0) \psi_c\not \in \Ran(F_\psi(0,0))$, where
\be N(F_\psi(0,0) )=\mbox{\rm span}\{\psi_c\}. \ee
\end{description}
Then there is a neighborhood $U$ of $(0,0)$ in $R\times X$, an interval $(-a,a)$, and continuous functions
\be  \kappa: (-a,a) \mapsto (-\infty,\infty),\,\,\,\,\,\,  \Psi: (-a,a) \mapsto X/N(F_\psi(0,0))
\ee
 such that  $\kappa(0)=0,\, \Psi(0) =0$
and
\be F^{-1}(0)\cap U = \Big\{ ( \kappa(\epsilon), \epsilon \psi_c +\epsilon \Psi(\epsilon)\Big|\, \,\,|\epsilon | <a \Big\} \cup \Big\{ (\tau,0) \Big|\, (\tau,0) \in U\Big\}.
\ee
\end{theorem}

\begin{theorem}
\label{main2}
 Assume that the spectral problem (\ref{LL}) admits a critical solution $(\lambda,\psi, \nu,\mu)=(0,\psi_c,\nu_c,\mu_c)$ for $\nu_c>0$,  $\mu_c\ge 0$.
 Then   there exist  continuous functions
\be  \kappa : (-a,a) \mapsto (-\infty,\infty),\,\,\, \Psi: (-a,a) \mapsto  H^4/\mbox{span}\{\psi_c\}\, \mbox{ with }\,\, \kappa(0)=0,\,\, \Psi(0)=0,
 \ee
for a small constant $a>0$,  so that  system (\ref{NS1})-(\ref{bdc}) has the  bifurcating steady-state solution $(\psi,\nu,\mu)$ expressed as
\bbe &&\psi=\psi_0+\epsilon \psi_{c} +\epsilon \Psi(\epsilon),\,\,\,\nu=\nu_c+\nu_c\kappa(\epsilon),\,\,\mu=\mu_c +\mu_c\kappa(\epsilon), \,\mbox{ for }\, \epsilon \in (-a,a).\label{dd1}
\label{dd1x}
\bee
\end{theorem}

It should be noted that Theorem \ref{main2} is  the same as Theorem \ref{main}, after we use the setting
\be \kappa(\epsilon) = \epsilon \delta(\epsilon,\psi_c), \,\,\,\Psi(\epsilon) = \epsilon \psi_i(\epsilon,\psi_c) \,\,\mbox{ for }\,\, 0\le \epsilon \le a,
\ee
\be \kappa(\epsilon) = -\epsilon \delta(-\epsilon,-\psi_c), \,\,\,\Psi(\epsilon) = \epsilon \psi_i(-\epsilon,-\psi_c) \,\,\mbox{ for } \,\,-a\le \epsilon <0.
\ee
\begin{proof}

 For employing Theorem \ref{tho}, we  formulate  the fluid motion problem into the functional framework of Theorem \ref{tho} and then show the validity of properties (a)-(d) and  (\ref{In}).
  Indeed, using the perturbation
  \bbe\label{pert}
  \psi=\psi_0+\psi',\,\, \nu= \nu_c+\tau\nu_c, \,\, \mu =\mu_c+\tau\mu_c,
  \bee
  we may rewrite (\ref{NS1})  as
\bbe  F(\tau,\psi)=0, \bee
for
\bbe F(\tau,\psi)= -(\tau+1)\mu_c \Delta \psi+ (\tau+1)\nu_c\Delta^2  \psi+ J(\psi_0,(2+\Delta)\psi) +J(\psi,\Delta \psi),\label{F}
 \label{LmL}
\bee
after omitting  the superscript prime.  Thus $(\tau,\psi)$ being a steady-state flow  means $(\tau,\psi)\in F^{-1}(0)$ and we are seeking steady-state solutions branching off the basic flow $(\tau,\psi)=(0,0)$.  It follows from Theorem \ref{main}(i) that
\bbe  \psi_c \in E_i\,\,\mbox{  for some integer }\,\, 1\le i \le 3. \label{my1} \bee
Let   $X=H^4_i$ and $Y=H_i$.  The injection property (\ref{In}) is valid due to   (\ref{inv}), and then the  definition given by  (\ref{F}) implies the validity of   properties (a)-(b).

Moreover, note that  $\nu_c>0$ and
\be F_\psi(0,0)\psi = -\mu_c \Delta \psi+ \nu_c\Delta^2  \psi+ J(\psi_0,(2+\Delta)\psi). \ee
The linear operator $F_{\psi}(0,0): X \mapsto Y$ is  Fredholm index zero. Hence the complement
\be Y/\Ran(F_\psi(0,0))= \mbox{\rm span}\{\psi_c^*\}\ee
for $\psi_c^*$ the conjugate counterpart of $\psi_c$, since
\be N(F_\psi (0,0)) =\mbox{\rm span}\{ \psi_c\}\ee
due to (\ref{one}) and (\ref{my1}). We thus have property (c).

 Additionally, upon the observation
\be F_{\tau\psi}(0,0)\psi = -\mu_c \Delta \psi+ \nu_c\Delta^2  \psi,\ee
equation  (\ref{two}) becomes   \bbe( F_{\tau\psi}(0,0)\psi_c, \psi_c^*)\ne 0 \,\mbox{ or }\,  F_{\tau\psi}(0,0)\psi_c\not \in \Ran(F_\phi(0,0)),
\bee
and hence property (d) holds true.
Here we have used the property
\be \Ran(F_\psi(0,0))= \Big\{\phi \in Y\Big| \, (\phi, \psi_c^*) =0\Big\}.\ee
Therefore, by Theorem \ref{tho}, we have the desired functions $\kappa$ and $\Psi$ so that the equation $F(\tau,\psi')=0$ has  solutions
\be \psi' = \epsilon \psi_c +\epsilon \Psi(\epsilon),\,\,\, \tau = \kappa(\epsilon)
\ee
This together with (\ref{pert}) implies Theorem \ref{main2}. The proof is completed.
\end{proof}

\section{Numerical computation}

{\sl Firstly}, we compute  critical  spectral solutions.
Let $\mathbf S$ be the set of all spectral solutions  $(\psi_c,\nu_c,\mu_c)$ with $\nu_c,\,\mu_c\ge 0$ and $(\nu_c,\mu_c)\ne (0,0)$ satisfying  the spectral problem
\bbe 0&=&-\mu_c  \Delta \psi_c+ \nu_c\Delta^2  \psi_c+ J(\psi_0,(2+\Delta)\psi_c),\label{LLl}
\\0&\ne &\psi_c= \sum_{n,m\ge 1} a_{n,m}\sin \frac {nx}2\sin \frac{my}2 \in H^4.\nonumber
\bee
By  Theorem \ref{th2},  the  set of all critical values $\{(\nu_c,\mu_c)\}$  is the union of the following three subsets
\bbe S_i =\Big\{ (\mu_c, \nu_c) \,\Big| \, (\psi_c,\nu_c,\mu_c)\in {\mathbf S}\,\mbox{ and }\,  \psi_c \in E_i \Big\},\,\,\, i=1,2,3.
\bee
Numerical spectral solutions are obtained by using the MATLAB eig function. $S_1$ is the same as the set of numerical data
 in \cite[Table I]{Thess} and is the  curve  joining  the points $(0.2371, 0)$ and $(0, 0.2310)$ in Figure \ref{ff3} (a), which shows that $S_1$ is  the neutral line separating  the linear stable and linear unstable domains.
 However, for the eigenfunction $\psi_c\in E_2$, the proof of Theorem \ref{th2}  shows the coexistence of two critical orthogonal eigenfunctions. The other one  is given by (\ref{hat}).
   The numerical simulation of the two eigenfunctions sharing  with the same critical vector value $(\nu_c,\mu_c)$  was given in \cite{Chen2019}.
  In fact, we have  $S_2=S_3$, since horizontal coordinate is symmetric with the vertical coordinate in the spectral problem (\ref{LLl}).
   The subset $S_2=S_3$ forms the line  touching the points $(0.0415, 0)$ and $(0, 0.01515)$ inside the linear unstable domain displayed in Figure \ref{ff3}(a). For displaying purpose, following three  classes of approximating critical eigenfunctions
 \bbe \psi_c \approx  \sum_{ n,m \mbox{ \footnotesize odd};\, 1\le n,m \le 11} a_{n,m} \sin \frac{nx}2 \sin\frac{mx}2 \in E_1, \,\,\, \label{eig1}\bee
\bbe \psi_c \approx  \sum_{ n \mbox{ \footnotesize odd}; \, m \mbox{ \footnotesize even};\, 1\le n,m \le 11} a_{n,m} \sin \frac{nx}2 \sin\frac{mx}2\in E_2, \,\,\, \label{eig2}\bee
\bbe \psi_c \approx  \sum_{ n \mbox{ \footnotesize even}; \, m \mbox{ \footnotesize odd};\, 1\le n,m \le 11} a_{n,m} \sin \frac{nx}2 \sin\frac{mx}2\in E_3, \,\,\,\label{eig3}\bee
 at some   typical critical values in $S_1\cup S_2\cup S_3$,  are exhibited respectively  in Figure \ref{ff3}(b)-(d).
\begin{figure}[h!]
 \centering
 \includegraphics[height=.35\textwidth, width=.35\textwidth]{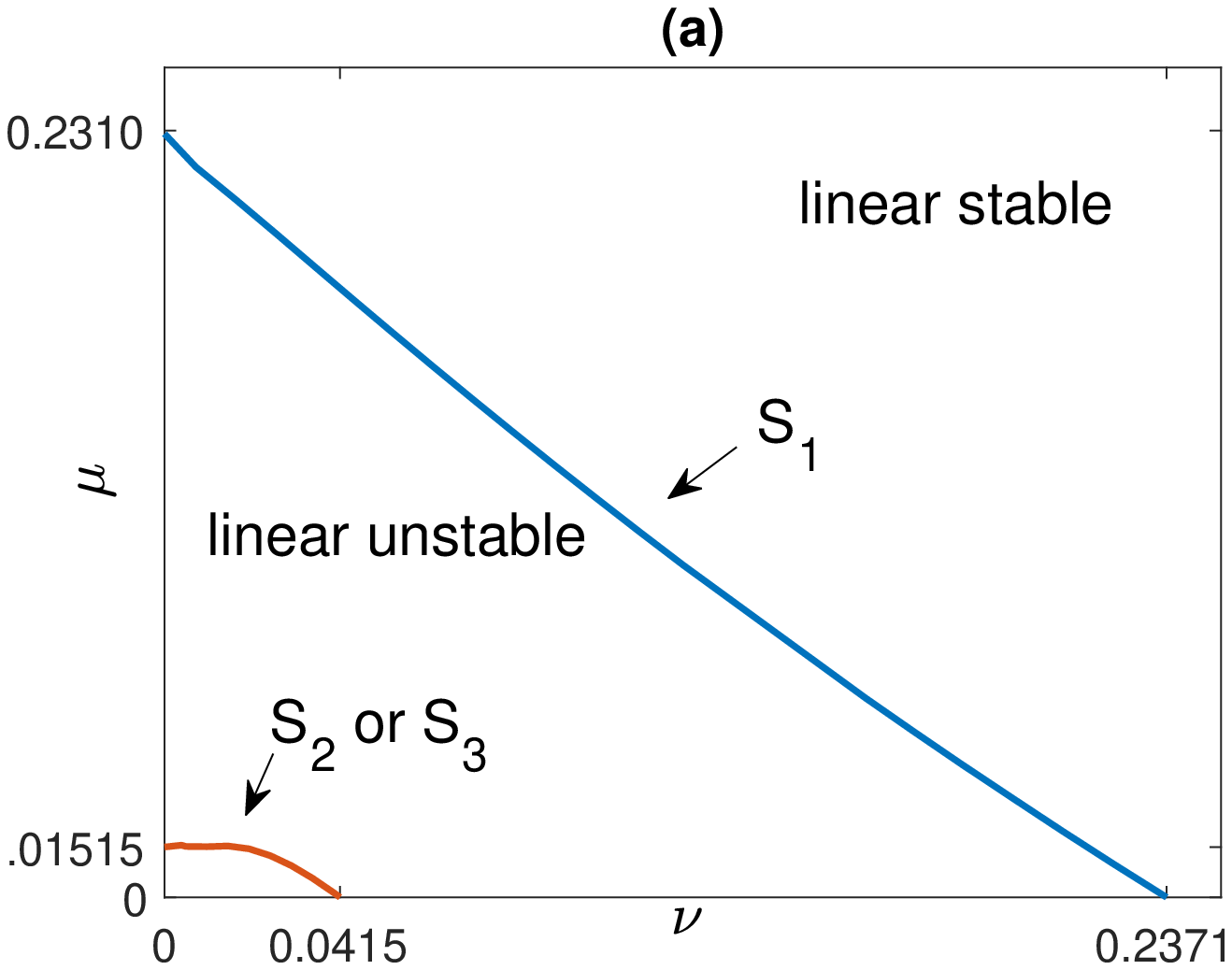}
\includegraphics[height=.35\textwidth, width=.35\textwidth]{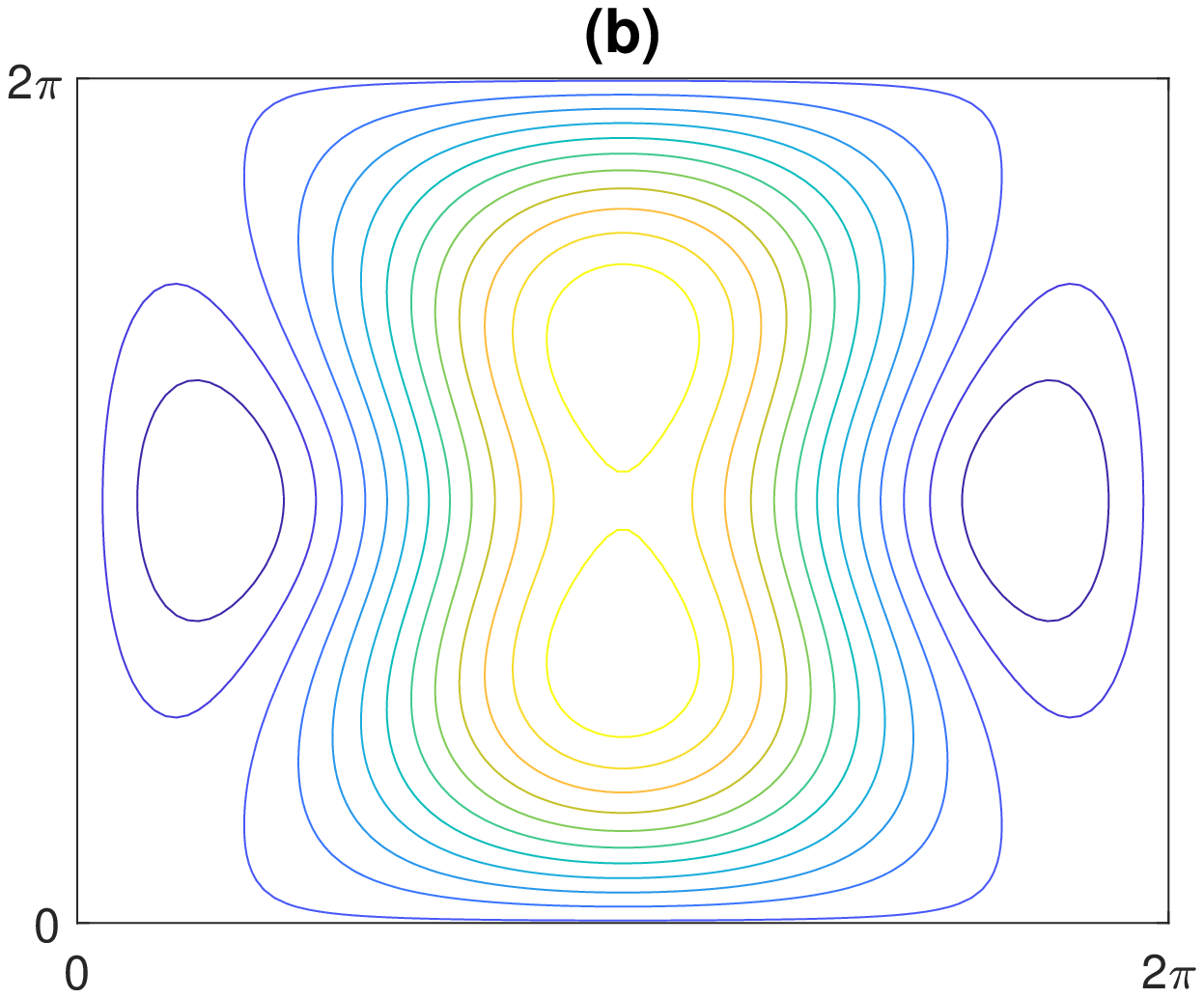}
\includegraphics[height=.35\textwidth, width=.35\textwidth]{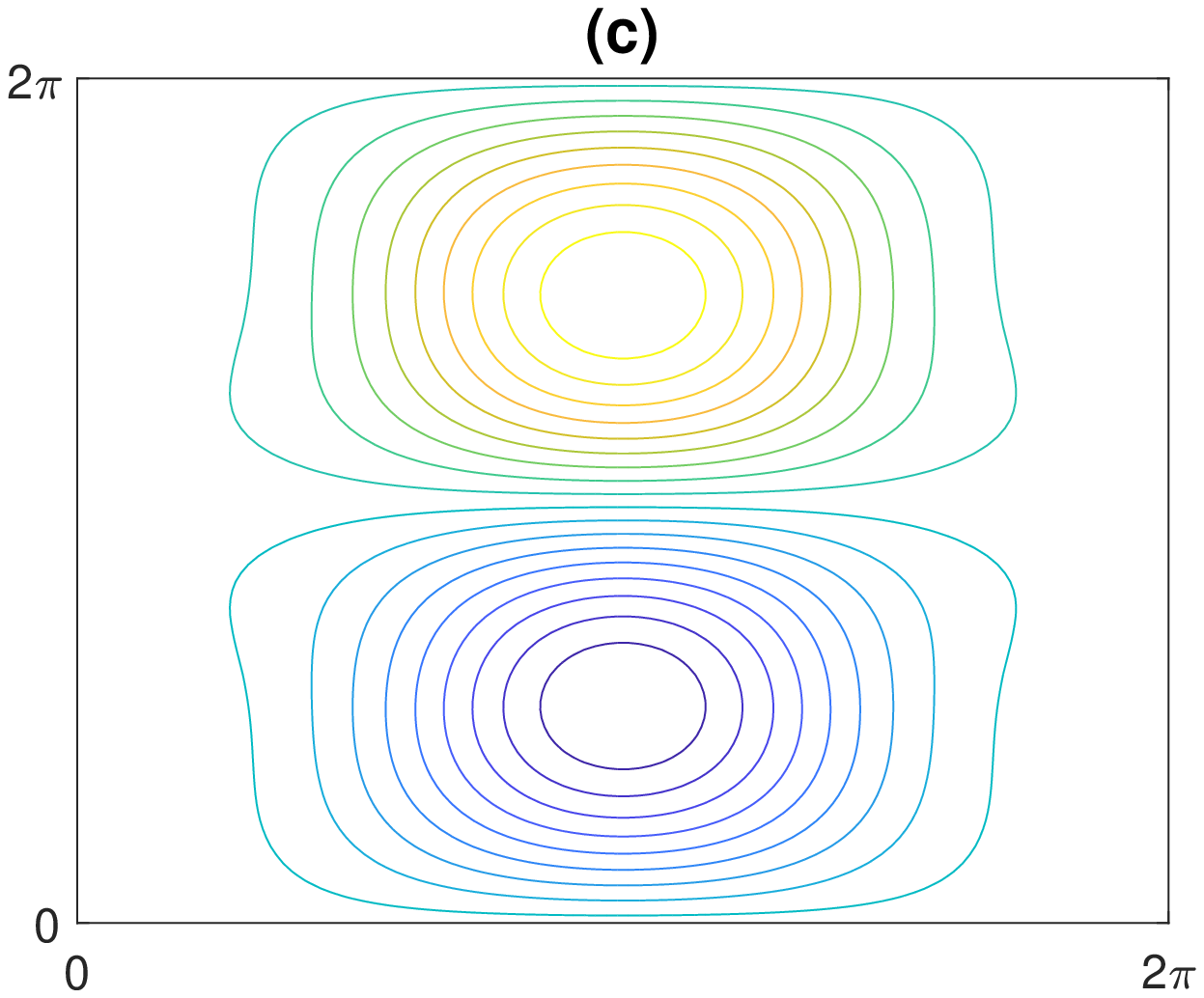}
\includegraphics[height=.35\textwidth, width=.35\textwidth]{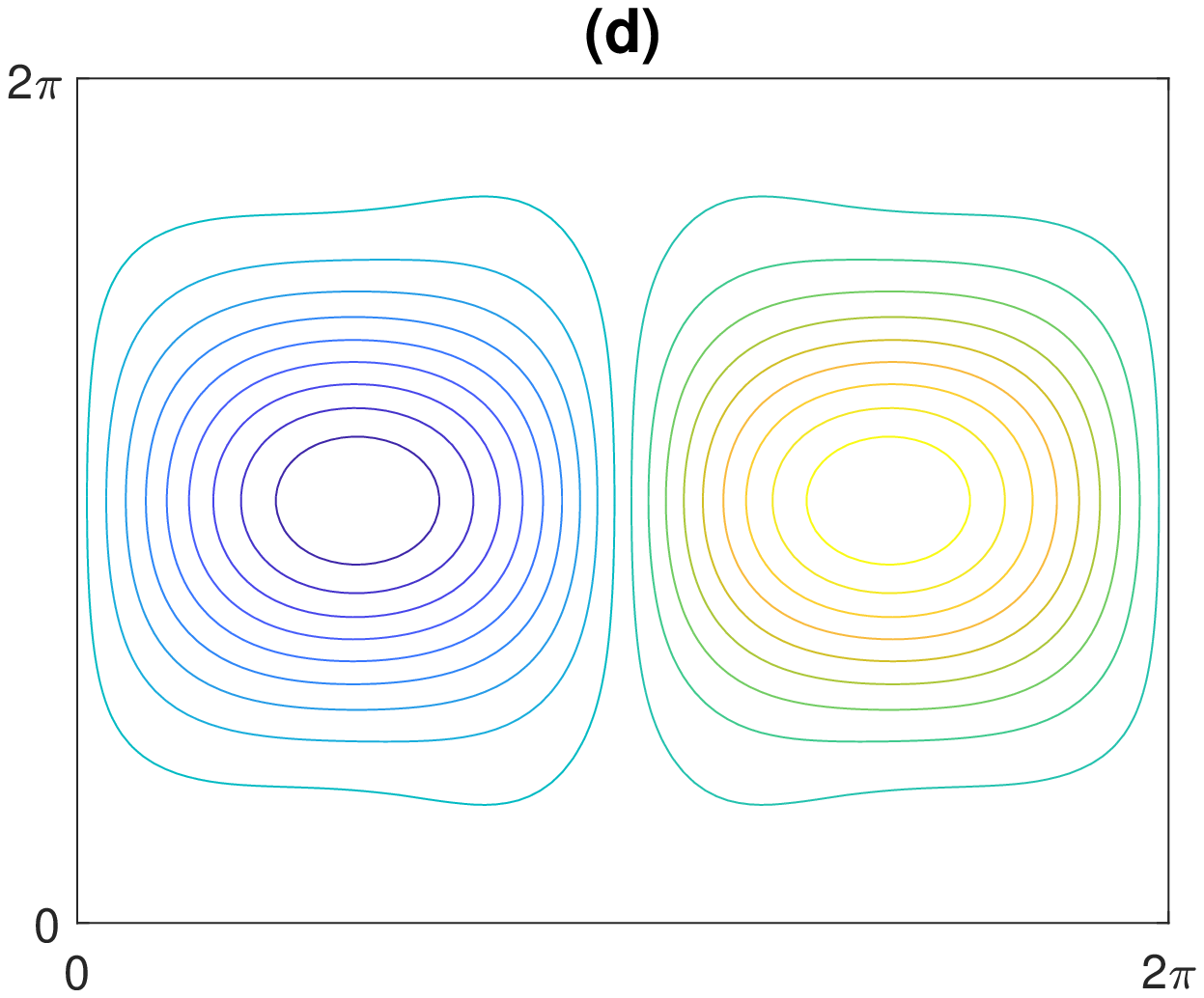}
\includegraphics[height=.35\textwidth, width=.35\textwidth]{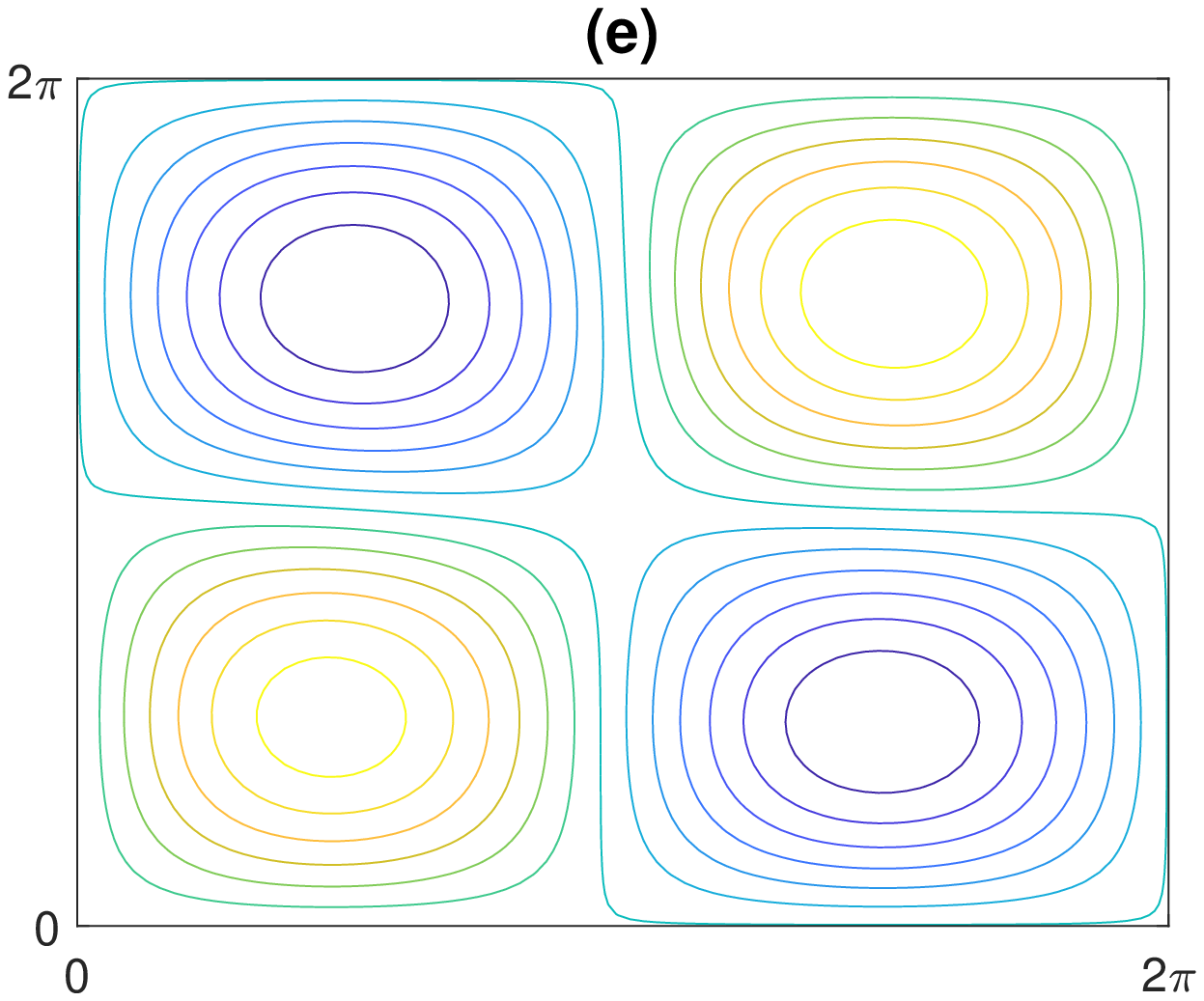}
 \caption{(a) All critical values $(\nu_c,\mu_c)\in S_1\cup S_2\cup S_3$;
 (b) critical eigenfunction $\psi_{c}\in E_1$ for  $(\nu_c,\mu_c)=(0.00054, 0.2315)$ or $(Re_c,Rh_c)=(22446, 1.326)$;
  (c) critical eigenfunction $\psi_{c}\in E_2$ for  $(\nu_c,\mu_c)=(0.00054,0.0177)$ or $(Re_c,Rh_c)=(6378, 4.929)$;
  (d) critical eigenfunction $\psi_{c}\in E_3$ for  $(\nu_c,\mu_c)=(0.00054,0.0177)$ or $(Re_c,Rh_c)=(6378, 4.929)$;
 (e)  numerical presentation of the nonlinear secondary flow at $(\nu,\mu)=(0.0005, 0.23)$.}

\label{ff3}
 \end{figure}
 \begin{figure}[h!]
 \centering
 \includegraphics[height=.35\textwidth, width=.35\textwidth]{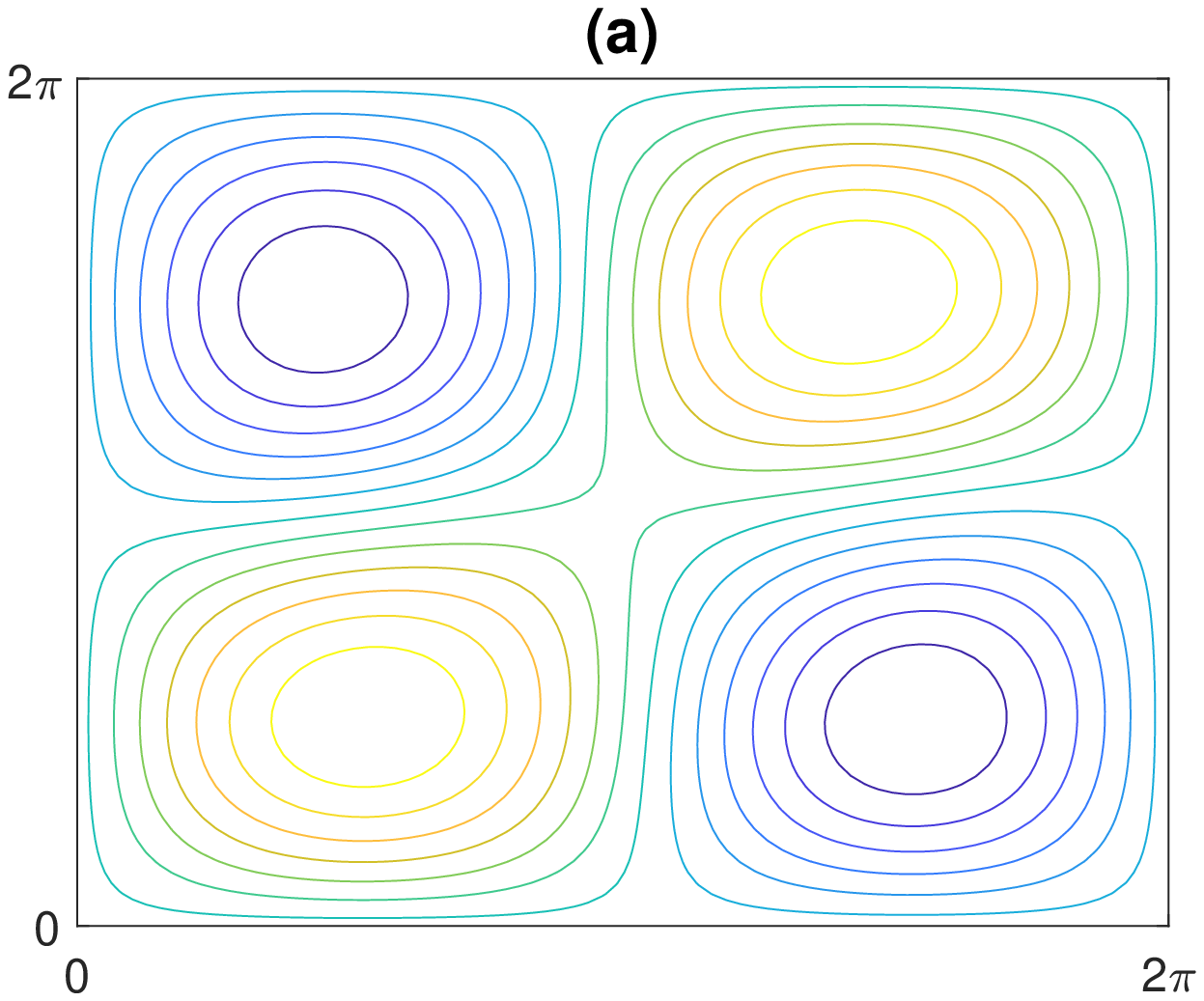}
\includegraphics[height=.35\textwidth, width=.35\textwidth]{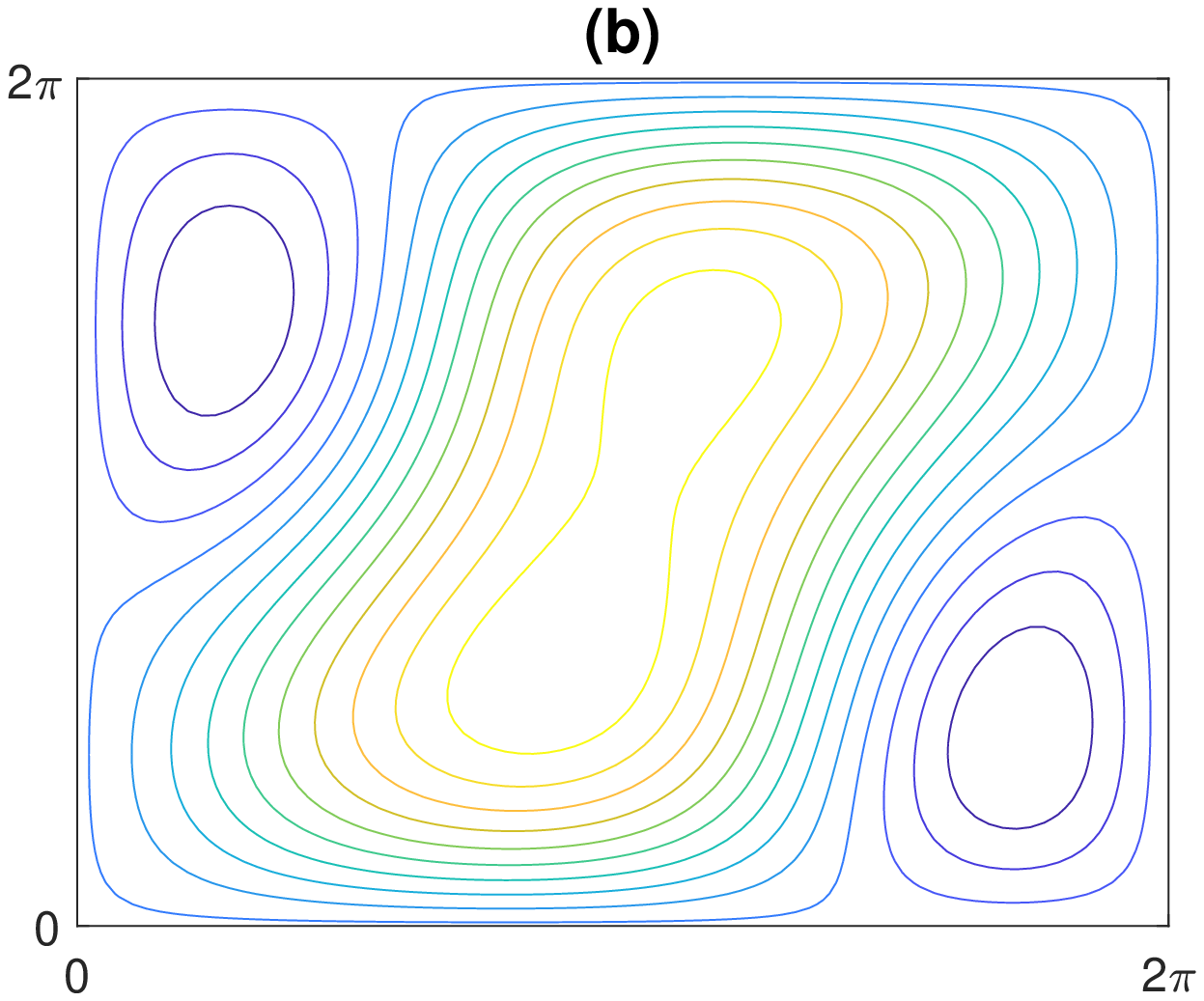}
\includegraphics[height=.35\textwidth, width=.35\textwidth]{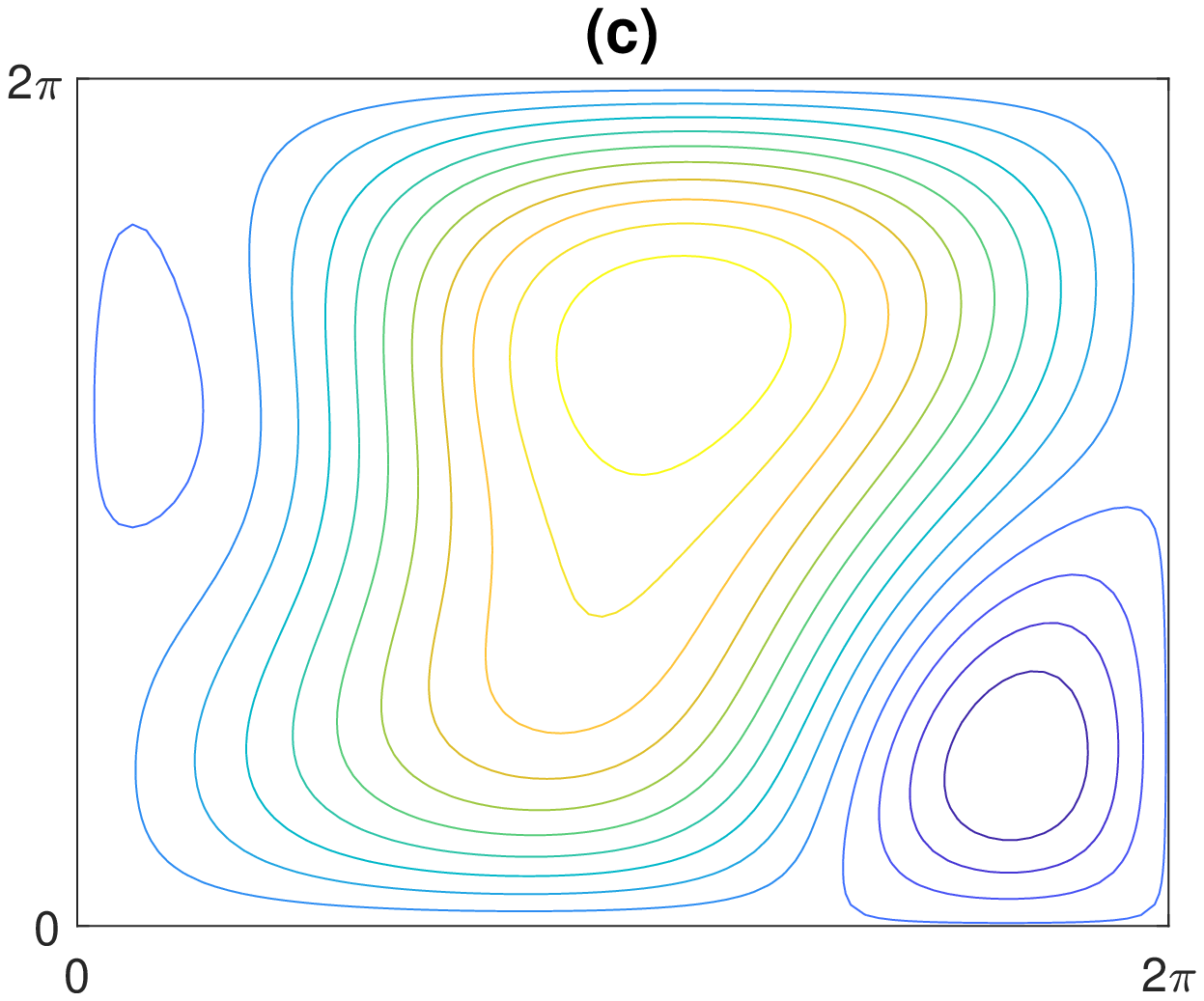}
\includegraphics[height=.35\textwidth, width=.35\textwidth]{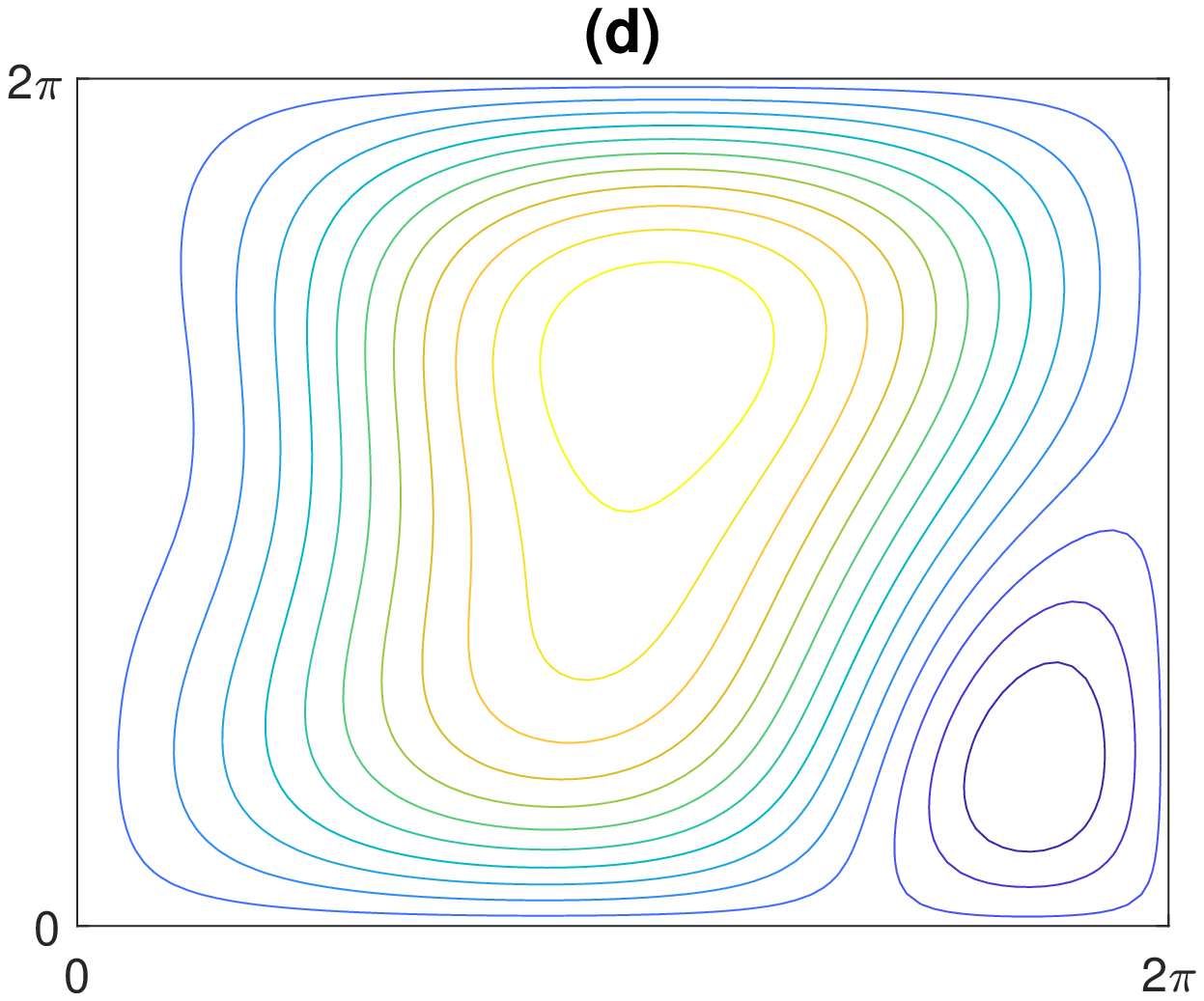}
 \caption{
 The  steady-state flows   when (a)   $(Re,Rh)=(700,1.55)$, (b)   $(Re,Rh)=(700,3)$, (c)   $(Re,Rh)=(700,5)$, and (d)   $(Re,Rh)=(700,7)$.
 }

\label{ff4}
 \end{figure}

{\sl Secondly}, we compute the bifurcating steady-state  flow shown in Theorem \ref{main}.
  To understand the experimental magnetohydrodynamic flows, we follow   \cite{Sommeria86,Sommeria,Sommeria87} to consider almost inviscid flows so that their energy dissipation is essentially controlled by the Hartmann layer friction $\mu$ or the Rayleigh number $Rh$. When the critical eigenfunction $\psi_c\in E_2\cup E_3$, the secondary flows branching from the corresponding critical vector values in the linear unstable domain   is unobservable in laboratory experiments, although they are contributed to the complexity of flow dynamic behaviour towards to turbulence.  We only consider the flows related to $(\psi_c,\nu_c,\mu_c)$ with $\psi_c\in E_1$.

If a secondary flow bifurcating in the direction $\psi_c\in E_1$ is  stable,  it attracts flows initially from the  states in its  vicinity and thus can be reached   by numerical computation.
Indeed, the desired  stable secondary flow is obtained by  employing  a finite difference scheme with a $80\times 80$ gridding mesh of the fluid domain $\Omega$.  For example, a numerical secondary solution is  obtained for $(\nu,\mu)$ close to the critical condition  $(\nu_c,\mu_c)=(0.00054,0.2306)$ or $(Re_c,Rh_c)=(22402, 1.329)$. In Figure \ref{ff3}(e), we present nonlinear secondary steady-state flow  at $(\nu,\mu)=(0.0005, 0.23)$ or $(Re,Rh)=(24158, 1.33)$, which represents the secondary  flow bifurcating from $\psi_0$  at $(\nu_c,\mu_c)=(0.00054,0.2306)$. This secondary flow is actually the limit of  the flow initially from the state (see figure \ref{ff1} (b))
\bbe  \psi_0 - 0.1 \sin \frac x2\sin \frac y2. \label{init}\bee
Therefore, the steady-state flow is obtained by computing the non-stationary flow in the  numerical computation.


 The secondary flow in Figure \ref{ff3}(e) shows the topological transition for the merging of two vortices, observed by Sommaria and Verron \cite{Sommeria, Sommeria87}. Their experimental  threshold for the onset of secondary flow  is $Rh_c=1.52$, which is close to  but higher than the present numeric one $Rh_c=1.329$. This is due to the neglect of the energy dissipation inside the lateral boundary layers of the original three-dimensional fluid motion problem (see \cite{Thess}).

 When $(\nu,\mu)$ is  close to the threshold $(\nu_c,\mu_c)$, the nonlinear secondary flow in Figure \ref{ff3}(e) is comparable  with the initial  form (\ref{init})
 expressed in Figure \ref{ff1} (b). This is owing to the principal mode $\sin \frac x2\sin \frac y2$  generating  the eigenfunction $\psi_c$. By numerical computation and (\ref{aaa2}), the principal  coefficient $a_{1,1}$ of the principal mode is significantly larger than other coefficients $a_{n,m}$.

{\sl Finally}, we  show  an additional nonlinear   topological bifurcation  of the flow motion by   considering  another  bifurcating flow  with respect to the initial state
\bbe  \psi_0 + 0.1 \sin \frac x2\sin \frac y2. \label{ini}\bee
By choosing a  smaller Reynolds number value, the initial state leads to the occurrence of the four steady-state flows  shown respectively in Figure \ref{ff4} (a)-(d) in four different $(Re,Rh)$ values. Their corresponding $(\nu,\mu)$ values are $(0.0173, 0.1982)$, $(0.0102,0.0603)$, $(0.0072,0.0254)$ and $(0.0058, 0.0148)$, respectively. Figure \ref{ff4} shows the inverse energy cascade towards large scales by merging four vortices into three and then combining three vortices into two.


\end{document}